\documentclass[11pt]{article}

\usepackage[latin1]{inputenc}
\usepackage[english]{babel}
\usepackage{here,enumerate,graphicx}
\usepackage{amsmath,amssymb,amscd,amsthm,mathrsfs}
\usepackage[colorlinks=true,linkcolor=blue,urlcolor=blue]{hyperref}

\usepackage[flushmargin]{footmisc}
\newcommand\blfootnote[1]{%
	\begingroup
	\renewcommand\thefootnote{}\footnote{#1}%
	\addtocounter{footnote}{-1}%
	\endgroup}

\theoremstyle{plain}
\newtheorem{theorem}{Theorem}[section]
\newtheorem{proposition}[theorem]{Proposition}

\newtheorem{lemma}[theorem]{Lemma}
\newtheorem{corollary}[theorem]{Corollary}

\theoremstyle{definition}
\newtheorem{definition}[theorem]{Definition}
\newtheorem{remark}[theorem]{Remark}
\newtheorem{example}[theorem]{Example}
\newtheorem{question}[theorem]{Question}

\newtheorem{problem}[theorem]{Problem}

\numberwithin{equation}{section}

\newcommand{\NN}{\mathbb{N}}

\newcommand{\II}{\mathbb{I}}

\newcommand{\Ac}{\mathcal{A}}
\newcommand{\Bc}{\mathcal{B}}
\newcommand{\Cc}{\mathcal{C}}
\newcommand{\Fc}{\mathcal{F}}
\newcommand{\Kc}{\mathcal{K}}
\newcommand{\Pc}{\mathcal{P}}
\newcommand{\Tc}{\mathcal{T}}
\newcommand{\dsup}{\overline{\operatorname{dens}}}
\newcommand{\dinf}{\underline{\operatorname{dens}}}
\newcommand{\cl}{\overline}
\newcommand{\diam}{\operatorname{diam}}
\newcommand{\orb}{\mathcal{O}}
\newcommand{\lspan}{\operatorname{span}}
\newcommand{\eps}{\varepsilon}

\newcommand{\Prox}{\operatorname{Prox}}
\newcommand{\D}[1]{\text{D}#1$\tfrac{\text{1}}{\text{2}}$}
\newcommand{\DC}[1]{\text{DC}#1$\tfrac{\text{1}}{\text{2}}$}

\newcommand{\com}{\overline}%{\bar} % Nos decidimos entre "{\bar}" y "{\overline}"
\newcommand{\fuz}{\hat}%{\widehat} % Nos decidimos entre "{\hat}" y "{\widehat}"
\newcommand{\eend}{\operatorname{end}}
\newcommand{\send}{\operatorname{send}}

\begin{document}
\begin{center}
	\begin{LARGE}
		{\bf Li-Yorke chaos on fuzzy dynamical systems}
	\end{LARGE}
\end{center}

%\vspace*{0.1cm}

\begin{center}
	\begin{Large}
		by
	\end{Large}
\end{center}

%\vspace*{0.1cm}

\begin{center}
	\begin{Large}
		Illych \'Alvarez \& Antoni L\'opez-Mart\'inez\blfootnote{\textbf{Journal-ref}: Information Sciences, Volume 749, article number 123527, (2026).\\ \textbf{DOI}: \href{https://doi.org/10.1016/j.ins.2026.123527}{https://doi.org/10.1016/j.ins.2026.123527}}
	\end{Large}
\end{center}

\vspace*{0.2cm}

\begin{abstract}
	Given a dynamical system $(X,f)$ we investigate how several variants of Li-Yorke chaos behave with respect to the extended systems $(\mathcal{K}(X),\overline{f})$ and $(\mathcal{F}(X),\hat{f})$, where~$\overline{f}$ is the hyperextension of~$f$ acting on the space $\mathcal{K}(X)$ of non-empty compact subsets of $X$, and where~$\hat{f}$ denotes the Zadeh extension of~$f$ acting on the space $\mathcal{F}(X)$ of normal fuzzy subsets of $X$. We first prove that the main variants of Li-Yorke chaos transfer from $(X,f)$ to $(\mathcal{K}(X),\overline{f})$ and from $(\mathcal{K}(X),\overline{f})$~to~$(\mathcal{F}(X),\hat{f})$, but that the converse implications do not hold in general. However, combining the notions of proximality and sensitivity we introduce Cantor-dense Li-Yorke chaos, and we prove that this strengthened variant of chaos does transfer from $(\mathcal{F}(X),\hat{f})$ to $(\mathcal{K}(X),\overline{f})$ under natural assumptions.
\end{abstract}

\vspace*{0.2in}

\section{Introduction}

A \textit{dynamical system} is a pair $(X,f)$ formed by a continuous map $f \colon X \longrightarrow X$ acting on a topological space $(X,\tau)$, usually called the \textit{phase space}. The study of dynamical systems is fundamental in many real-world contexts, since any process that evolves over time can be modelled as a dynamical system: given an initial state of the system $x_0 \in X$ one can study its \textit{$f$-orbit}
\[
\mathcal{O}_f(x_0) := \{ x_0, f(x_0), f^2(x_0), \dots \} = \{ f^n(x_0) \ ; \ n \geq 0 \},
\]
where $f^n(x_0)$ represents the state of the system after $n$ time steps. In order to quantify deviations or distances between states, it is common to assume that $(X,\tau)$ is endowed with a metric $d$, turning it into a metric space $(X,d)$. However, in many realistic situations, the exact state of a system may not be fully determined or accessible. In these cases, it is natural to study the evolution of \textit{sets} or \textit{fuzzy sets} rather than individual points, being this (fuzzy) sets a usually good enough representation of statistical possible states of the system (see \cite{Kupka2011_IS_on}). This perspective, often referred to as \textit{collective dynamics}, emphasizes the interplay between the \textit{individual dynamics} of $(X,f)$ and the induced dynamics on the hyperspaces of non-empty compact sets $(\Kc(X),\com{f})$ and fuzzy sets $(\Fc(X),\fuz{f})$. Understanding how dynamical properties transfer from the base system to its hyperspaces extensions is a question of both theoretical and potential applied relevance. In fact, motivated by recent developments at the interface between dynamical systems, optimisation, and control theory (see for instance \cite{SusantoSuHa2025_RCO_LQR}, where a fuzzy hyperspace viewpoint relied on previous theoretical results such as those in \cite{AlvarezLoPe2025_FSS_recurrence}), in this paper we focus on providing a rigorous theoretical framework for the dynamical behaviour of the systems $(X,f)$, $(\Kc(X),\com{f})$ and $(\Fc(X),\fuz{f})$ with respect to Li-Yorke chaotic properties. While our work is primarily theoretical, we anticipate that it may inform future applications, as has already occurred with \cite{AlvarezLoPe2025_FSS_recurrence}.

As we were advancing in the previous paragraph, our main interest in this paper lies in trying to understand the interplay between the Li-Yorke chaotic-type properties exhibited by $(X,f)$ and those presented by its induced extensions, namely $(\Kc(X),\com{f})$ and $(\Fc(X),\fuz{f})$. Here, $\com{f}$ denotes the usual {\em hyperextension} of $f$, acting on the hyperspace $\Kc(X)$ of non-empty compact subsets of $X$, while $\fuz{f}$ stands for the so-called {\em Zadeh extension} (or {\em fuzzification}) of $f$, acting on the space $\Fc(X)$ of normal fuzzy sets of $X$. Given a metric space $(X,d)$ we will endow the hyperspace $\Kc(X)$ with the usual {\em Hausdorff metric}~$d_H$, while~$\Fc(X)$ will be equipped with four distinguished metrics: the {\em supremum metric}~$d_{\infty}$, the {\em Skorokhod metric}~$d_{0}$, the {\em sendograph metric}~$d_{S}$ and the {\em endograph metric}~$d_{E}$. For each $\rho \in \{ d_{\infty}, d_{0}, d_{S}, d_{E} \}$ we will denote the corresponding metric space $(\Fc(X),\rho)$ by $\Fc_{\infty}(X)$, $\Fc_{0}(X)$, $\Fc_{S}(X)$ and $\Fc_{E}(X)$, respectively. Using this fuzzy framework, the second author of this paper proved in his recent works \cite{Lopez2026_IJFS_topological-I,Lopez2026_JIA_topological-II} that:
\begin{enumerate}[--]
	\item If $\Pc$ denotes any of the dynamical notions of {\em topological $\Ac$-transitivity}, {\em topological $(\ell,\Ac)$-recurrence}, {\em Devaney chaos} or the {\em specification property}, then \textbf{all} the dynamical systems $(\Kc(X),\com{f})$, $(\Fc_{\infty}(X),\fuz{f})$, $(\Fc_{0}(X),\fuz{f})$, $(\Fc_{S}(X),\fuz{f})$, and $(\Fc_{E}(X),\fuz{f})$ exhibit $\Pc$ whenever \textbf{any} of them does (see \cite{Lopez2026_IJFS_topological-I}).
	
	\item However, for the so-called {\em contractive}, {\em expansive}, {\em expanding}, {\em positively expansive}, {\em chain recurrence} and the {\em shadowing} properties, it was proved that $(\Kc(X),\com{f})$ exhibits each of these properties if and only if so does $(\Fc_{\infty}(X),\fuz{f})$, but that extremely radical behaviours are presented by $(\Fc_{E}(X),\fuz{f})$, and sometimes by $(\Fc_{0}(X),\fuz{f})$ and $(\Fc_{S}(X),\fuz{f})$, with respect to these dynamical notions (see \cite{Lopez2026_JIA_topological-II}).
\end{enumerate}
Both papers \cite{Lopez2026_IJFS_topological-I,Lopez2026_JIA_topological-II} were motivated by the references \cite{Kupka2011_IS_on} and \cite{AlvarezLoPe2025_FSS_recurrence,BartollMaPeRo2022_AXI_orbit,JardonSan2021_FSS_expansive,JardonSan2021_IJFS_sensitivity,JardonSanSan2020_FSS_some,JardonSanSan2020_MAT_transitivity,MartinezPeRo2021_MAT_chaos}, where the aforementioned dynamical properties were initially studied for fuzzy systems. Since the notions of Li-Yorke and uniform distributional chaos were considered in \cite{MartinezPeRo2021_MAT_chaos} but not addressed in \cite{Lopez2026_IJFS_topological-I,Lopez2026_JIA_topological-II}, the aim of this paper is to investigate the main variants of Li-Yorke chaos for the system $(\Fc(X),\fuz{f})$. As we shall see throughout this paper, although $(\Kc(X),\com{f})$ and $(\Fc_{\infty}(X),\fuz{f})$ behave in parallel with respect to all the properties studied in \cite{Lopez2026_IJFS_topological-I,Lopez2026_JIA_topological-II}, they differ when it comes to Li-Yorke-type chaotic properties. This observation explains why Li-Yorke chaos and its variants were not considered in \cite{Lopez2026_IJFS_topological-I,Lopez2026_JIA_topological-II} and, for completeness, we provide a summary of the theory developed in those papers together with the results of the present work in Section~\ref{Sec_6:conclusions}.

Focusing on Li-Yorke-type chaotic properties, which were originally studied in \cite{MartinezPeRo2021_MAT_chaos} in the context of fuzzy dynamical systems, we must start by mentioning that in \cite[Proposition~3]{MartinezPeRo2021_MAT_chaos} it was shown that the notions of Li-Yorke and uniform distributional chaos transfer from $(X,f)$ to $(\Kc(X),\com{f})$, but also from $(\Kc(X),\com{f})$ to the extensions $(\Fc_{\infty}(X),\fuz{f})$ and $(\Fc_{0}(X),\fuz{f})$. For the reader's convenience, we recall the precise result established there:

\begin{proposition}[\textbf{\cite[Proposition~3]{MartinezPeRo2021_MAT_chaos}}]
	Let $f$ be a continuous map on a metric space $X$. Then:
	\begin{enumerate}[{\em(i)}]
		\item If there exists a ($\eps$-distributionally) scrambled set $S$ for $f$, then there exist ($\eps$-distributionally) scrambled sets $\com{S}$ and $\fuz{S}$ for $\com{f}$ and $\fuz{f}$, respectively, with the same cardinality as $S$.
		
		\item If there exists a ($\eps$-distributionally) scrambled set $\com{S}$ for $\com{f}$, then there exists a ($\eps$-distributionally) scrambled set $\fuz{S}$ for $\fuz{f}$ with the same cardinality as $\com{S}$.
		
		\item If the map $f$ is Li-Yorke (uniformly distributionally) chaotic on $X$, then $\com{f}$ is Li-Yorke (uniformly distributionally) chaotic on the hyperspace $\Kc(X)$.
		
		\item If the map $\com{f}$ is Li-Yorke (uniformly distributionally) chaotic on the hyperspace $\Kc(X)$, then $\fuz{f}$ is Li-Yorke (uniformly distributionally) chaotic on the spaces of fuzzy sets $\Fc_{\infty}(X)$ and on $\Fc_{0}(X)$.
	\end{enumerate}
\end{proposition}

In addition, the following question was indirectly posed in \cite[Remark~1]{MartinezPeRo2021_MAT_chaos}:

\begin{question}\label{Ques:Li-Yorke.transfer}
	Is there a dynamical system $(X,f)$ for which $(\Fc_{\infty}(X),\fuz{f})$ or $(\Fc_{0}(X),\fuz{f})$ are Li-Yorke or uniformly distributionally chaotic but such that $(\Kc(X),\com{f})$ and hence $(X,f)$ are not?
\end{question}

Although in \cite{MartinezPeRo2021_MAT_chaos} the authors considered only two types of Li-Yorke chaos (the classical one and uniform distributional chaos), many other notions of this kind exist. Moreover, the metrics $d_{S}$ and $d_{E}$, and consequently the associated dynamical systems $(\Fc_{S}(X),\fuz{f})$ and $(\Fc_{E}(X),\fuz{f})$, were not addressed in~\cite{MartinezPeRo2021_MAT_chaos}. Motivated by these facts, in this paper we extend \cite[Proposition~3]{MartinezPeRo2021_MAT_chaos} to mean Li-Yorke chaos and to several variants of distributional chaos (namely, DC1, \DC{1}, DC2, \DC{2}, and DC3). We consider all the metrics $d_{\infty}$, $d_{0}$, $d_{S}$, and $d_{E}$; see Theorem~\ref{The:scrambled} below. 

In this paper, we also provide a more sophisticated extension of \cite[Proposition~3]{MartinezPeRo2021_MAT_chaos}, showing that the extended fuzzy system $(\Fc(X),\fuz{f})$ may exhibit very strong Li-Yorke-type chaotic properties under extremely weak assumptions on $(X,f)$; see Theorem~\ref{The:pairs}. As an application, we obtain a strong positive answer to Question~\ref{Ques:Li-Yorke.transfer}; see Example~\ref{Exa_1:main}. Furthermore, by studying the notions of proximality and sensitivity (to which Section~\ref{Sec_4:sensitivity} is devoted), we establish several new results concerning what we call \emph{Cantor-dense Li-Yorke chaos}; see Theorems~\ref{The:cantor} and~\ref{The:operators}.

The paper is organized as follows. In Section~\ref{Sec_2:notation} we introduce the variants of Li-Yorke chaos that we consider, together with the necessary background on fuzzy sets. In Section~\ref{Sec_3:Li-Yorke} we present our main results, extending \cite[Proposition~3]{MartinezPeRo2021_MAT_chaos} in two directions (see Theorems~\ref{The:scrambled} and~\ref{The:pairs}) and addressing questions of the type posed in Question~\ref{Ques:Li-Yorke.transfer} (see Examples~\ref{Exa_1:main},~\ref{Exa_2:E_MLYC}, and~\ref{Exa_3:E_DC3}). In Section~\ref{Sec_4:sensitivity} we analyze the notions of proximality and sensitivity. These are then used in Section~\ref{Sec_5:Cantor} to show that Question~\ref{Ques:Li-Yorke.transfer} admits a negative answer for \emph{Cantor-dense Li-Yorke chaos} on complete metric spaces and for certain linear operators (see Theorems~\ref{The:cantor} and~\ref{The:operators}). Finally, we conclude with a section summarising the main contributions of the paper and outlining possible directions for future research (see Section~\ref{Sec_6:conclusions}).

\section{General background and some key lemmas}\label{Sec_2:notation}

In this paper we combine the theory of dynamical systems with that of hyperspaces of compact and fuzzy sets. In this section, we start by recalling the definition of the main variants of Li-Yorke chaos, namely mean Li-Yorke and distributional chaos. Then we recall the definitions of $\Kc(X)$ and $\Fc(X)$, of the maps $\com{f}$ and $\fuz{f}$, and of the several metrics that we consider on such spaces. From now on let~$\NN$ be the set of strictly positive integers, set $\NN_0 := \NN \cup \{0\}$, and let $\II$ be the unit interval $[0,1]$.

\subsection{Li-Yorke chaos and its main variants}\label{SubSec_2.1:LY.MLY.DC}

The initial concept of chaos for dynamical systems was introduced by Li and Yorke in~\cite{LiYorke1975_AMM_period}, where they considered the next objects with respect to a continuous map $f:X\longrightarrow X$ on a metric space $(X,d)$:
\begin{enumerate}[--]
	\item A pair $(x,y) \in X\times X$ is said to be a {\em Li-Yorke pair} ({\em LY pair} for short), for the map $f$ and the metric~$d$, if there exists some $\eps>0$ such that
	\[
	\limsup_{n\to\infty} d(f^n(x),f^n(y)) \geq \eps \quad \text{ while } \quad \liminf_{n\to\infty} d(f^n(x),f^n(y)) = 0.
	\]
	In this case one also says that $(x,y)$ is a {\em Li-Yorke $\eps$-pair} ({\em LY $\eps$-pair} for short).
	
	\item A subset $S \subset X$ is said to be {\em LY scrambled} (resp.\ {\em LY $\eps$-scrambled}), for the map $f$ and the metric~$d$ (and resp.\ for some $\eps>0$), if for every pair of distinct points $x\neq y$ with $x,y\in S$ we have that the pair $(x,y)$ is a LY pair (resp.\ LY $\eps$-pair) for the map $f$ and the metric $d$ (and resp.\ for $\eps>0$).
	
	\item A dynamical system $(X,f)$ is said to be (resp.\ {\em uniformly}) {\em Li-Yorke chaotic}, for the metric $d$, if there exists an uncountable set $S \subset X$ that is LY scrambled (resp.\ LY $\eps$-scrambled) for the map~$f$ and the metric~$d$ (and resp.\ for some $\eps>0$). For short, we will write that $(X,f)$ is {\em LYC} (resp.\ {\em U-LYC}).
\end{enumerate}
Several alternative but closely related notions of chaos have been introduced to capture the complexity of dynamical systems from different points of view. For instance we have the concept of mean Li-Yorke chaos, which, although studied earlier under different names, was formally coined in~\cite{HuangLiYe2014_JFA_stable}:
\begin{enumerate}[--]
	\item A pair $(x,y) \in X\times X$ is said to be a {\em mean Li-Yorke pair} ({\em MLY pair} for short), for the map $f$ and the metric~$d$, if there exists some $\eps>0$ such that
	\[
	\limsup_{n\to\infty} \frac{1}{n} \sum_{j=1}^n d(f^j(x),f^j(y)) \geq \eps \quad \text{ while } \quad \liminf_{n\to\infty} \frac{1}{n} \sum_{j=1}^n d(f^j(x),f^j(y)) = 0.
	\]
	In this case one also says that $(x,y)$ is a {\em mean Li-Yorke $\eps$-pair} ({\em MLY $\eps$-pair} for short).
	
	\item A subset $S \subset X$ is said to be {\em MLY scrambled} (resp.\ {\em MLY $\eps$-scrambled}), for the map $f$ and the metric~$d$ (and resp.\ for some $\eps>0$), if for every pair of distinct points $x\neq y$ with $x,y\in S$ we have that $(x,y)$ is a MLY pair (resp.\ MLY $\eps$-pair) for the map $f$ and the metric $d$ (and resp.\ for $\eps>0$).

	\item A system $(X,f)$ is said to be (resp.\ {\em uniformly}) {\em mean Li-Yorke chaotic}, for the metric~$d$, if there exists an uncountable set $S \subset X$ that is MLY scrambled (resp.\ MLY $\eps$-scrambled) for~$f$ and the metric~$d$ (and resp.\ for some $\eps>0$). For short, we will write that $(X,f)$ is {\em MLYC} (resp.\ {\em U-MLYC}).
\end{enumerate}
It is not hard to check that every MLY $\eps$-pair is a LY $\eps$-pair, so that MLYC implies LYC and the same holds between their uniform variants. Another generalization of LYC is that of distributional chaos, whose type~1 version was introduced by Schweizer and Sm\'ital~\cite{SchweizerS1994_TAMS_measures} to characterize maps with positive topological entropy. In \cite{SmitalS2004_CSF_distributional} and \cite{BalibreaSS2005_CSF_the-three} distributional chaos of types~2~and~3 were introduced, and in~\cite{DolezelovaRothRoth2016_IJBC_on}~and~\cite{Hantakova2017_IJBC_iteration} two intermediate notions were proposed (types~1$\tfrac{1}{2}$ and~2$\tfrac{1}{2}$). To define them, let us recall that the {\em lower} and {\em upper densities} of a subset $A \subset \NN$ are defined as the quantities
\[
\dinf(A) := \liminf_{n\to\infty} \frac{\#(A \cap [1,n])}{n} \quad \text{ and } \quad \dsup(A) := \limsup_{n\to\infty} \frac{\#(A \cap [1,n])}{n}.
\]
Following the above references, given a continuous map $f:X\longrightarrow X$ on a metric space $(X,d)$ and a pair $(x,y)\in X\times X$, the {\em lower distributional function} generated by $f$ for the pair $(x,y)$ is the function
\[
\Phi_{(x,y)}(\delta) := \dinf \bigl(\{ j \in \NN \ ; \ d(f^j(x),f^j(y))<\delta \}\bigr) \quad \text{ for each } \delta>0,
\]
and the {\em upper distributional function} generated by $f$ for the pair $(x,y)$ is the function
\[
\Phi^*_{(x,y)}(\delta) := \dsup \bigl(\{ j \in \NN \ ; \ d(f^j(x),f^j(y))<\delta \}\bigr) \quad \text{ for each } \delta>0.
\]
With these notations, following \cite{SchweizerS1994_TAMS_measures,DolezelovaRothRoth2016_IJBC_on,Hantakova2017_IJBC_iteration,JiangLi2025_JMAA_chaos}:
\begin{enumerate}[--]
	\item A pair $(x,y) \in X\times X$ is said to be a {\em distributional pair of type 1} ({\em D1 pair} for short), for the map~$f$ and the metric~$d$, if there exists some $\eps>0$ such that $\Phi_{(x,y)}(\eps)=0$ while $\Phi^*_{(x,y)}(\delta)=1$ for all $\delta>0$. In this case one also says that $(x,y)$ is a {\em distributional $\eps$-pair of type 1} ({\em D1 $\eps$-pair} for short).
	
	\item A subset $S \subset X$ is said to be {\em D1 scrambled} (resp.\ {\em D1 $\eps$-scrambled}), for the map $f$ and the metric~$d$ (and resp.\ for some $\eps>0$), if for every pair of distinct points $x\neq y$ with $x,y\in S$ we have that the pair $(x,y)$ is a D1 pair (resp.\ D1 $\eps$-pair) for the map $f$ and the metric $d$ (and resp.\ for $\eps>0$).
	
	\item A system $(X,f)$ is said to be (resp.\ {\em uniformly}) {\em distributionally chaotic of type 1}, for the metric~$d$, if there exists an uncountable set $S \subset X$ that is D1 scrambled (resp.\ D1 $\eps$-scrambled) for~$f$ and the metric~$d$ (and resp.\ for some $\eps>0$). For short, we will write that $(X,f)$ is {\em DC1} (resp.\ {\em U-DC1}).
\end{enumerate}
For type 1$\tfrac{\text{1}}{\text{2}}$, as far as we know, the uniform version that we are about to define has not been considered in the literature. Nevertheless, following \cite{DolezelovaRothRoth2016_IJBC_on,Hantakova2017_IJBC_iteration}, we have that:
\begin{enumerate}[--]
	\item A pair $(x,y) \in X\times X$ is said to be a {\em distributional pair of type 1$\tfrac{\text{1}}{\text{2}}$} ({\em \D{1} pair} for short), for the map~$f$ and the metric~$d$, if $\lim_{\delta\to0^+} \Phi_{(x,y)}(\delta)=0$ while $\Phi^*_{(x,y)}(\delta)=1$ for all $\delta>0$.
	
	\item A subset $S \subset X$ is said to be {\em \D{1} scrambled}, for the map $f$ and the metric~$d$, if for every pair of distinct points $x\neq y$ with $x,y\in S$ we have that the pair $(x,y)$ is a \D{1} pair for the map $f$ and the metric $d$. In addition, we will say that the set $S$ is {\em \D{1} U-scrambled} if it is \D{1} scrambled and
	\[
	\lim_{\delta\to0^+} \sup\{ \Phi_{(x,y)}(\delta) \ ; \ x,y \in S \text{ with } x\neq y \} = 0.
	\]
	
	\item A system $(X,f)$ is said to be (resp.\ {\em uniformly}) {\em distributionally chaotic of type 1$\tfrac{\text{1}}{\text{2}}$}, for the metric~$d$, if there exists an uncountable set $S \subset X$ that is \D{1} scrambled (resp.\ \D{1} U-scrambled) for the map~$f$ and the metric~$d$. For short, we will write that $(X,f)$ is {\em \DC{1}} (resp.\ {\em U-\DC{1}}).
\end{enumerate}
For type 2, following \cite{SmitalS2004_CSF_distributional,DolezelovaRothRoth2016_IJBC_on,Hantakova2017_IJBC_iteration,JiangLi2025_JMAA_chaos}:
\begin{enumerate}[--]
	\item A pair $(x,y) \in X\times X$ is said to be a {\em distributional pair of type 2} ({\em D2 pair} for short), for~$f$ and the metric~$d$, if there exists some $0<\eps\leq 1$ such that $\Phi_{(x,y)}(\eps) \leq 1-\eps$ while $\Phi^*_{(x,y)}(\delta)=1$ for all $\delta>0$. In this case one also says that $(x,y)$ is a {\em distributional $\eps$-pair of type 2} ({\em D2 $\eps$-pair} for short).
	
	\item A subset $S \subset X$ is said to be {\em D2 scrambled} (resp.\ {\em D2 $\eps$-scrambled}), for the map $f$ and the metric~$d$ (and resp.\ for some $0<\eps\leq 1$), if for every pair of distinct points $x\neq y$ with $x,y\in S$ we have that the pair $(x,y)$ is a D2 pair (resp.\ D2 $\eps$-pair) for~$f$ and the metric $d$ (and resp.\ for $0<\eps\leq 1$).
	
	\item A system $(X,f)$ is said to be (resp.\ {\em uniformly}) {\em distributionally chaotic of type 2}, for the metric~$d$, if there exists an uncountable set $S \subset X$ that is D2 scrambled (resp.\ D2 $\eps$-scrambled) for~$f$ and the metric~$d$ (and resp.\ for some $\eps>0$). For short, we will write that $(X,f)$ is {\em DC2} (resp.\ {\em U-DC2}).
\end{enumerate}
For type 2$\tfrac{\text{1}}{\text{2}}$, as in the case of type~1$\tfrac{1}{2}$, and as far as we know, the uniform version that we are about to define has not been considered in the literature. Nevertheless, following \cite{DolezelovaRothRoth2016_IJBC_on,Hantakova2017_IJBC_iteration}, we have that:
\begin{enumerate}[--]
	\item A pair $(x,y) \in X\times X$ is said to be a {\em distributional pair of type 2$\tfrac{\text{1}}{\text{2}}$} ({\em \D{2} pair} for short), for the map~$f$ and the metric~$d$, if there exist two positive values $\eps,c>0$ such that
	\[
	\Phi_{(x,y)}(\delta) < c < \Phi^*_{(x,y)}(\delta) \quad \text{ whenever } 0 < \delta \leq \eps.
	\]
	In this case we will say that $(x,y)$ is a {\em distributional $(\eps,c)$-pair of type 2$\tfrac{\text{1}}{\text{2}}$} ({\em \D{2} $(\eps,c)$-pair} for short).
	
	\item A subset $S \subset X$ is said to be {\em \D{2} scrambled} (resp.\ {\em \D{2} $(\eps,c)$-scrambled}), for~$f$ and the metric~$d$ (and resp.\ for some $\eps,c>0$), if for every pair of distinct points $x\neq y$ with $x,y\in S$ we have that $(x,y)$ is a \D{2} pair (resp.\ \D{2} $(\eps,c)$-pair) for the map $f$ and the metric $d$ (and resp.\ for $\eps,c>0$).
	
	\item A system $(X,f)$ is said to be (resp.\ {\em uniformly}) {\em distributionally chaotic of type 2$\tfrac{\text{1}}{\text{2}}$}, for the metric~$d$, if there exists an uncountable set $S \subset X$ that is \D{2} scrambled (resp.\ \D{2} $(\eps,c)$-scrambled) for~$f$ and~$d$ (and resp.\ for some $\eps,c>0$). For short, we will write that $(X,f)$ is {\em \DC{2}} (resp.\ {\em U-\DC{2}}).
\end{enumerate}
For type 3, following \cite{BalibreaSS2005_CSF_the-three,DolezelovaRothRoth2016_IJBC_on,Hantakova2017_IJBC_iteration}:
\begin{enumerate}[--]
	\item A pair $(x,y) \in X\times X$ is said to be a {\em distributional pair of type 3} ({\em D3 pair} for short), for the map~$f$ and the metric~$d$, if there exist two non-negative numbers $0\leq a < b$ such that
	\[
	\Phi_{(x,y)}(\delta) < \Phi^*_{(x,y)}(\delta) \quad \text{ whenever } a < \delta < b.
	\]
	In this case we will say that $(x,y)$ is a {\em distributional $(a,b)$-pair of type 3} ({\em D3 $(a,b)$-pair} for short).
	
	\item A subset $S \subset X$ is said to be {\em D3 scrambled} (resp.\ {\em D3 $(a,b)$-scrambled}), for the map~$f$ and the metric~$d$ (and resp.\ for some $0\leq a<b$), if for every pair of distinct points $x\neq y$ with $x,y\in S$ we have that $(x,y)$ is a D3 pair (resp.\ D3 $(a,b)$-pair) for~$f$ and the metric $d$ (and resp.\ for $0\leq a<b$).
	
	\item A system $(X,f)$ is said to be (resp.\ {\em uniformly}) {\em distributionally chaotic of type 3}, for the metric~$d$, if there exists an uncountable set $S \subset X$ that is D3 scrambled (resp.\ D3 $(a,b)$-scrambled) for~$f$ and~$d$ (and resp.\ for some $0\leq a<b$). For short, we will write that $(X,f)$ is {\em DC3} (resp.\ {\em U-DC3}).
\end{enumerate}
It is well-known that DC1 $\Rightarrow$ \DC{1} $\Rightarrow$ DC2 $\Rightarrow$ \DC{2} $\Rightarrow$ DC3, that \DC{2} $\Rightarrow$ LYC, but that~DC3 and~LYC do not imply each other (see \cite{DolezelovaRothRoth2016_IJBC_on}). The same holds for the uniform variants. In addition, it was observed in~\cite{Downarowicz2014_PAMS_positive} that~MLYC is equivalent to DC2 whenever the underlying metric space $(X,d)$ is bounded (see also \cite[Remark~5.5]{JiangLi2025_JMAA_chaos}). The following is a modest extension of such an observation:

\begin{lemma}\label{Lem:MLYC<->DC2}
	Let $f:X\longrightarrow X$ be a continuous map acting on a metric space $(X,d)$, and let $x,y \in X$ be a pair of points fulfilling that $\{ d(f^j(x),f^j(x)) \ ; \ j \in \NN \}$ is bounded by some $r>0$. Hence:
	\begin{enumerate}[{\em(a)}]
		\item If $(x,y)$ is a MLY $\eps$-pair for some $\eps>0$, then $(x,y)$ is a D2 $\tfrac{\eps}{r+1}$-pair.
		
		\item If $(x,y)$ is a D2 $\eps$-pair for some $0<\eps\leq 1$, then $(x,y)$ is a MLY $\eps^2$-pair.
	\end{enumerate}
	In particular, we have that $(x,y)$ is a MLY pair if and only if $(x,y)$ is a D2 pair.
\end{lemma}
\begin{proof}
	Let us start noticing that for each $\delta>0$ and each $n \in \NN$ we have that
	\begin{equation}\label{eq:delta.n.D<M}
		\delta \cdot \frac{1}{n} \#\{ 1\leq j\leq n \ ; \ d(f^j(x),f^j(y)) \geq \delta \} \leq \frac{1}{n} \sum_{j=1}^n d(f^j(x),f^j(y)).
	\end{equation}
	Moreover, using the bound $r>0$, for each $\delta>0$ and each $n \in \NN$ it can be easily checked that
	\begin{equation}\label{eq:delta.n.r.M<D}
		\frac{1}{n} \sum_{j=1}^n d(f^j(x),f^j(y)) \leq \delta + r \cdot \frac{1}{n} \#\{ 1\leq j\leq n \ ; \ d(f^j(x),f^j(y)) \geq \delta \}.
	\end{equation}
	From now on we only have to use that $\dsup(A)=1-\dinf(\NN\setminus A)$ for every subset $A \subset \NN$. Actually:
	\begin{enumerate}[(a)]
		\item If $(x,y)$ is a MLY $\eps$-pair, taking superior limits in \eqref{eq:delta.n.r.M<D} for $\delta=\tfrac{\eps}{r+1}$ implies $\Phi_{(x,y)}(\tfrac{\eps}{r+1}) \leq 1-\tfrac{\eps}{r+1}$, while taking inferior limits in \eqref{eq:delta.n.D<M} for each $\delta>0$ shows that $\Phi^*_{(x,y)}(\delta)=1$ for all $\delta>0$.
		
		\item If $(x,y)$ is a D2 $\eps$-pair, taking superior limits in \eqref{eq:delta.n.D<M} for $\delta=\eps$ implies that
		\[
		\limsup_{n\to\infty} \frac{1}{n} \sum_{j=1}^n d(f^j(x),f^j(y)) \geq \eps(1-\Phi_{(x,y)}(\eps)) \geq \eps^2,
		\]
		while taking inferior limits in \eqref{eq:delta.n.r.M<D} shows that $0 \leq \liminf_{n\to\infty} \frac{1}{n} \sum_{j=1}^n d(f^j(x),f^j(y)) \leq \delta$ for each positive $\delta>0$, and hence that this inferior limit is equal to $0$.\qedhere
	\end{enumerate}
\end{proof}

By Lemma~\ref{Lem:MLYC<->DC2}, the notions of U-MLYC and U-DC2 are equivalent for every dynamical system fulfilling the suitable boundedness conditions. This fact will be repeatedly used in Section~\ref{Sec_3:Li-Yorke}.

\subsection{The hyperspaces of compact and fuzzy sets}

Given a metric space $(X,d)$ we will denote by $\Bc_d(x,\eps) \subset X$ the open $d$-ball centred at $x \in X$ and of radius $\eps>0$. Moreover, we will consider the spaces $\Cc(X) :=\left\{ C \subset X \ ; \ C \text{ is a non-empty closed set} \right\}$ and $\Kc(X) := \left\{ K \subset X \ ; \ K \text{ is a non-empty compact set} \right\}$. Given two sets $C_1,C_2 \in \Cc(X)$, the value
\[
\textstyle d_H(C_1,C_2) := \max\left\{ \sup_{x_1 \in C_1} \inf_{x_2 \in C_2} d(x_1,x_2) , \sup_{x_2 \in C_2} \inf_{x_1 \in C_1} d(x_2,x_1) \right\}
\]
is the {\em Hausdorff distance} between $C_1$ and $C_2$. In $\Kc(X)$, the map~$d_H:\Kc(X)\times\Kc(X)\longrightarrow[0,+\infty[$ defined as above is a metric, called the {\em Hausdorff metric}, and every continuous map $f:X\longrightarrow X$ induces a $d_H$-continuous map $\com{f}:\Kc(X)\longrightarrow\Kc(X)$, defined as~$\com{f}(K) := f(K) = \{ f(x) \ ; \ x \in K \}$ for each $K \in \Kc(X)$. We will denote by $\Bc_H(K,\eps) \subset \Kc(X)$ the open $d_H$-ball centred at $K \in \Kc(X)$ and of radius $\eps>0$. Given $Y \subset X$ and $\eps\geq 0$ we write $Y+\eps := \{ x \in X \ ; \ d(x,y) \leq \eps \text{ for some } y \in Y \}$. The next fact is then well-known, and we refer the reader to \cite{IllanesNad1999_book_hyperspaces} for more on $\Cc(X)$ and $\Kc(X)$:

\begin{proposition}\label{Pro:Hausdorff}
	Let $(X,d)$ be a metric space, $\eps\geq 0$, and let $C_1,C_2,C_3,C_4 \in \Cc(X)$. Then:\\[-20pt]
	\begin{enumerate}[{\em(a)}]
		\item We have that $d_H(C_1,C_2) \leq \eps$ if and only if $C_1 \subset C_2+\eps$ and $C_2 \subset C_1+\eps$.\\[-15pt]
		
		\item We always have that $d_H(C_1\cup C_2, C_3 \cup C_4) \leq \max\{ d_H(C_1,C_3) , d_H(C_2,C_4) \}$.
	\end{enumerate}
\end{proposition}

A {\em fuzzy~set} on the metric space $(X,d)$ will be a function $u:X\longrightarrow\II$, where the value $u(x) \in \II$ denotes the degree of membership of $x$ in $u$. For such a set $u$, its {\em $\alpha$-level} is the set
\[
u_{\alpha} := \{ x \in X \ ; \ u(x) \geq \alpha \} \ \text{ for each } \alpha \in \ ]0,1] \quad \text{ and } \quad u_0 := \cl{\{ x \in X \ ; \ u(x)>0 \}}.
\]
The symbol $\Fc(X)$ will stand for the space of normal fuzzy sets of $(X,d)$, i.e.,\ that formed by the fuzzy sets~$u$ that are upper-semicontinuous functions and such that~$u_0$ is compact and~$u_1$ is non-empty. For each $K \in \Kc(X)$, its characteristic function $\chi_K:X\longrightarrow\II$ belongs to $\Fc(X)$. The piecewise-constant function $u := \max_{1\leq l\leq N} \alpha_l\chi_{K_l}$ also belongs to $\Fc(X)$ whenever $0 < \alpha_1 < \alpha_2 < ... < \alpha_N = 1$ and we choose $K_1,K_2,...,K_N \in \Kc(X)$. Moreover, in \cite{JardonSanSan2020_FSS_some,JardonSanSan2020_MAT_transitivity,MartinezPeRo2021_MAT_chaos} it was proved (or used) the next:

\begin{lemma}\label{Lem:eps.pisos}
	Let $(X,d)$ be a metric space. Given any normal fuzzy set $u \in \Fc(X)$ and any $\eps>0$ there exist numbers $0 = \alpha_0 < \alpha_1 < \alpha_2 < ... < \alpha_N = 1$ fulfilling that $d_H(u_0,u_{\alpha_1})<\eps$ and that $d_H(u_{\alpha}, u_{\alpha_{l+1}})<\eps$ for every level $\alpha \in \ ]\alpha_l, \alpha_{l+1}]$ with $0\leq l\leq N-1$.
\end{lemma}

Given a dynamical system $(X,f)$, it follows from \cite[Propositions~3.1 and 4.9]{JardonSanSan2020_FSS_some} that there exists a well-defined map $\fuz{f}:\Fc(X)\longrightarrow\Fc(X)$, called the {\em Zadeh extension} of $f$, which maps each normal fuzzy set $u \in \Fc(X)$ to the normal fuzzy set $\fuz{f}(u):X\longrightarrow\II$ with
\[
\left[\fuz{f}(u)\right](x) := \sup\{ u(y) \ ; \ y \in f^{-1}(\{x\}) \}, \quad \text{ if } f^{-1}(\{x\}) \neq \varnothing,
\]
and $\left[\fuz{f}(u)\right](x) := 0$, if $f^{-1}(\{x\}) = \varnothing$. Moreover, in \cite{JardonSanSan2020_FSS_some,JardonSanSan2020_MAT_transitivity,RomanChal2008_CSF_some} it was proved (or used) the next:

\begin{proposition}\label{Pro:fuz{f}}
	Let $f:X\longrightarrow X$ be a continuous map acting on a metric space $(X,d)$, $u \in \Fc(X)$, $\alpha \in \II$, $n \in \NN_0$, and $K \in \Kc(X)$. Then:\\[-20pt]
	\begin{enumerate}[{\em(a)}]
		\item $\left[ \fuz{f}(u) \right]_{\alpha} = f(u_{\alpha}) = \com{f}(u_{\alpha})$, i.e.\ the $\alpha$-level of the $\fuz{f}$-image coincides with the $\com{f}$-image of the $\alpha$-level;\\[-10pt]
		
		\item $\left( \fuz{f} \right)^n = \widehat{f^n}$, i.e.\ the composition of $\fuz{f}$ with itself $n$ times coincides with the fuzzification of $f^n$;\\[-10pt]
		
		\item $\fuz{f}\left( \chi_K \right) = \chi_{f(K)} = \chi_{\com{f}(K)}$, i.e.\ the fuzzification $\fuz{f}$ is an extension of the hyperextension $\com{f}$.
	\end{enumerate}
\end{proposition}

To talk about the continuity of the Zadeh extension $\fuz{f}:\Fc(X)\longrightarrow\Fc(X)$, we will endow the space of normal fuzzy sets $\Fc(X)$ with four different metrics. In particular, given a metric space $(X,d)$, the so-called {\em supremum metric} $d_{\infty}:\Fc(X)\times\Fc(X)\longrightarrow[0,+\infty[$ is defined as
\[
d_{\infty}(u,v) := \sup_{\alpha\in\II} d_H(u_{\alpha},v_{\alpha}) \quad \text{  for each pair } u,v \in \Fc(X),
\]
where $d_H$ is the Hausdorff metric on $\Kc(X)$. The {\em Skorokhod metric} $d_{0}:\Fc(X)\times\Fc(X)\longrightarrow[0,+\infty[$ is defined for each pair $u,v \in \Fc(X)$ as
\[
d_{0}(u,v) := \inf\left\{ \eps>0 \ ; \ \text{there is } \xi \in \Tc \text{ such that } \sup_{\alpha\in\II} |\xi(\alpha)-\alpha| \leq \eps \text{ and } d_{\infty}(u,\xi\circ v) \leq \eps \right\},
\]
where $\Tc$ is the set of strictly increasing homeomorphisms of the unit interval $\xi:\II\longrightarrow\II$. To talk about the {\em endograph} and {\em sendograph metrics} we need the auxiliary metric $\com{d}$ on the product $X\times\II$ with
\[
\com{d}\left((x,\alpha),(y,\beta)\right) := \max\left\{ d(x,y) , |\alpha - \beta | \right\} \quad \text{ for each pair } (x,\alpha),(y,\beta) \in X\times\II.
\]
Given now any fuzzy set $u \in \Fc(X)$, the {\em endograph of $u$} is defined as the set
\[
\eend(u) := \left\{ (x,\alpha) \in X\times\II \ ; \ u(x) \geq \alpha \right\},
\]
and the {\em sendograph of $u$} is defined as $\send(u) := \eend(u) \cap (u_0\times\II)$. Then, for each pair $u,v \in \Fc(X)$ the~{\em sendograph metric}~$d_S(u,v)$ on~$\Fc(X)$ is the Hausdorff metric $\com{d}_H(\send(u),\send(v))$ on $\Kc(X\times\II)$, and the {\em endograph metric} $d_{E}(u,v)$ on $\Fc(X)$ is the Hausdorff distance $\com{d}_H(\eend(u),\eend(u))$ on $\Cc(X\times\II)$. Although $d_{E}$ is defined as a distance of closed but not necessarily compact sets on $X\times\II$, the fact that the supports of the fuzzy sets $u,v \in \Fc(X)$ are compact implies that $d_{E}(u,v)$ is well-defined.

Following \cite{Lopez2026_IJFS_topological-I,Lopez2026_JIA_topological-II} we will denote by $\Bc_{\infty}(u,\eps)$, $\Bc_{0}(u,\eps)$, $\Bc_{S}(u,\eps)$ and $\Bc_{E}(u,\eps)$ the open balls centred at $u \in \Fc(X)$ and of radius $\eps>0$ for each of the metrics $d_{\infty}$, $d_{0}$, $d_{S}$ and $d_{E}$. Moreover, for each metric $\rho \in \{ d_{\infty} , d_{0} , d_{S} , d_{E} \}$ we will denote the metric space $(\Fc(X),\rho)$ by $\Fc_{\infty}(X)$, $\Fc_{0}(X)$, $\Fc_{S}(X)$ and $\Fc_{E}(X)$ respectively. These metrics come from the general theory of Spaces of Fuzzy Sets and each of them has its own role and importance (see \cite[Section~2]{Lopez2026_IJFS_topological-I}). In this paper we will need some of the main relations among these metrics (see \cite{JardonSan2021_FSS_expansive,JardonSan2021_IJFS_sensitivity,JardonSanSan2020_FSS_some,Lopez2026_IJFS_topological-I} and the references therein for more details):

\begin{proposition}\label{Pro:fuzzy.metrics}
	Let $(X,d)$ be a metric space, $u,v \in \Fc(X)$, $K,L \in \Kc(X)$ and $x \in X$. Then:
	\begin{enumerate}[{\em(a)}]
		\item $d_{E}(u,v) \leq d_{S}(u,v) \leq d_{0}(u,v) \leq d_{\infty}(u,v)$ and $d_{E}(u,v) \leq 1$.
		
		\item $d_{S}(\chi_{\{x\}},u) = d_{0}(\chi_{\{x\}},u) = d_{\infty}(\chi_{\{x\}},u) = d_H(\{x\},u_0) = \max\{ d(x,y) \ ; \ y \in u_0 \}$.
		
		\item $d_{0}(\chi_{K},u) = d_{\infty}(\chi_{K},u) = \max\{ d_H(K,u_0) , d_H(K,u_1) \}$.
		
		\item $d_{E}(\chi_{K},\chi_{L})=\min\{ d_H(K,L) , 1 \}$ while $d_{S}(\chi_{K},\chi_{L}) = d_{0}(\chi_{K},\chi_{L}) = d_{\infty}(\chi_{K},\chi_{L})=d_H(K,L)$.
		
		\item $d_H(u_0,v_0) \leq d_{S}(u,v)$ and also $\max\{ d_H(u_0,v_0), d_H(u_1,v_1) \} \leq d_{0}(u,v)$.
	\end{enumerate}
\end{proposition}

The next lemma was the key fact used in \cite{Lopez2026_IJFS_topological-I} to get most of the results obtained there, but we will also use it here in Section~\ref{Sec_4:sensitivity}:

\begin{lemma}[\textbf{\cite[Lemma~2.4]{Lopez2026_IJFS_topological-I}}]\label{Lem:key}
	Let $(X,d)$ be a metric space, and assume that the sets $K \in \Kc(X)$ and $u \in \Fc(X)$ fulfill that $\delta := d_{E}(\chi_K,u) < \tfrac{1}{2}$. Then $d_H(K,u_{\alpha}) \leq \delta$ for every $\alpha \in \ ]\delta,1-\delta]$.
\end{lemma}

For the Li-Yorke chaotic-type properties we need to bound from above and from below the value of the underlying metric. Thus, the following extension of Lemma~\ref{Lem:key} will be very useful:

\begin{lemma}\label{Lem:key2}
	Let $(X,d)$ be a metric space, $K \in \Kc(X)$, $u \in \Fc(X)$, and let $\delta,\eps>0$. Hence:
	\begin{enumerate}[{\em(a)}]
		\item If $\delta=d_{S}(\chi_K,u)<1$, then we have that $d_H(K,u_{\alpha})\leq\delta$ for all $\alpha \in [0,1-\delta]$.
		
		\item If $\eps < \min\{ d_{S}(\chi_K,u) , 1 \}$, then we have that $d_H(K,u_{\alpha})>\eps$ for some $\alpha \in [0,1-\eps]$.
		
		\item If $\eps < \min\{ d_{E}(\chi_K,u) , \tfrac{1}{2} \}$, then we have that $d_H(K,u_{\alpha})>\eps$ for some $\alpha \in \ ]\eps,1-\eps]$.
	\end{enumerate}
\end{lemma}
Statement (a) is a ``sendograph version'' of Lemma~\ref{Lem:key}. Moreover, statements (c) and (b) can be interpreted, respectively, as a reciprocal for statement (a) and Lemma~\ref{Lem:key}.
\begin{proof}[Proof of Lemma~\ref{Lem:key2}]
	(a): Given any $\alpha \in [0,1-\delta]$ we must show that $d_H(K,u_{\alpha})\leq\delta$ or, equivalently, that we have the inclusions $K \subset u_{\alpha} + \delta$ and $u_{\alpha} \subset K + \delta$ (see Proposition~\ref{Pro:Hausdorff}).
	
	To check $K \subset u_{\alpha} + \delta$, since $d_{S}(\chi_K,u) = \delta$ we have that $\send(\chi_K) \subset \send(u) + \delta$, so that given $x \in K$ we have that $(x,1) \in \send(\chi_K)$ and that there exists $(y,\beta) \in \send(u)$ for which
	\[
	\com{d}\left( (x,1) , (y,\beta) \right) = \max\{ d(x,y) , |1-\beta| \} \leq \delta.
	\]
	This implies that $u(y) \geq \beta \geq 1-\delta \geq \alpha$ so that $y \in u_{\alpha}$. We deduce that $x \in \{y\} + \delta \subset u_{\alpha} + \delta$, and hence that $K \subset u_{\alpha} + \delta$. To check that $u_{\alpha} \subset K + \delta$ we only need to use that $d_H(v_0,w_0) \leq d_{S}(v,w)$ for every pair of fuzzy sets $v,w \in \Fc(X)$ as stated in part (e) of Proposition~\ref{Pro:fuzzy.metrics}. This implies that $d_H(K,u_0) \leq d_{S}(\chi_K,u) = \delta$, and by part (a) of Proposition~\ref{Pro:Hausdorff} we have that $u_{\alpha} \subset u_0 \subset K + \delta$.
	
	(b): Assume that $d_H(K,u_{\alpha})\leq\eps<1$ for all $\alpha \in [0,1-\eps]$ and let us prove that then $d_{S}(\chi_K,u)\leq\eps$ or, equivalently, that the inclusions $\send(u) \subset \send(\chi_K) + \eps$ and $\send(\chi_K) \subset \send(u) + \eps$ hold.
	
	The inclusion $\send(u) \subset \send(\chi_K) + \eps$ follows from the assumption $d_H(K,u_0)\leq\eps$ as it implies that $u_{\alpha} \subset u_0 \subset K + \eps$ for all $\alpha \in \II$. To check that $\send(\chi_K) \subset \send(u) + \eps$ pick any point $(x,\alpha) \in \send(\chi_K)$ and note that, since $x \in K$ and we assumed that $d_H(K,u_{1-\eps})\leq\eps$, then $K \subset u_{1-\eps} + \eps$ and there exists a point $y \in u_{1-\eps}$ such that $d(x,y)\leq \eps$. As $y \in u_{1-\eps} \subset u_{\beta}$ for all $\beta \in [0,1-\eps]$, we have that
	\[
	(y,\alpha) \in \send(u) \quad \text{with} \quad \com{d}( (x,\alpha) , (y,\alpha) ) = \max\{ d(x,y) , |\alpha-\alpha| \} = d(x,y) \leq \eps,
	\]
	if $\alpha \in [0,1-\eps]$, or we have that
	\[
	(y,1-\eps) \in \send(u) \quad \text{with} \quad \com{d}( (x,\alpha) , (y,1-\eps) ) = \max\{ d(x,y) , |\alpha-(1-\eps)| \} \leq \eps,
	\]
	if $\alpha \in [1-\eps,1]$. Thus, we have that $(x,\alpha) \in \send(u) + \eps$, and hence that $\send(\chi_K) \subset \send(u) + \eps$.
	
	(c): Assume that $d_H(K,u_{\alpha})\leq\eps<\tfrac{1}{2}$ for all $\alpha \in \ ]\eps,1-\eps]$ and let us prove that then $d_{E}(\chi_K,u)\leq\eps$ or, equivalently, that we have $\eend(u) \subset \eend(\chi_K) + \eps$ and $\eend(\chi_K) \subset \eend(u) + \eps$.
	
	We check that $\eend(u) \subset \eend(\chi_K) + \eps$. Indeed, our assumption implies that $u_{\alpha} \subset K + \eps$ for all $\alpha \in \ ]\eps,1-\eps]$, and then that $u_{\alpha} \subset u_{1-\eps} \subset K + \eps$ for all $\alpha \in [1-\eps,1]$. Moreover, for $(x,\alpha) \in \eend(u)$ with $\alpha \in [0,\eps]$ it follows that
	\[
	(x,0) \in \eend(\chi_K) \quad \text{with} \quad \com{d}( (x,\alpha) , (x,0) ) = \max\{ d(x,x) , |\alpha| \} \leq \eps.
	\]
	Altogether we get that $\eend(u) \subset \eend(\chi_K) + \eps$. To prove that $\eend(\chi_K) \subset \eend(u) + \eps$ one can argue similarly to the respective part of statement (b), where only $d_H(K,u_{1-\eps})\leq\eps$ is needed. In fact, the only difference with the proof of statement (b) is the case in which we pick $(x,\alpha) \in \eend(\chi_K)$ for the value $\alpha=0$, but in this situation we trivially have that $(x,\alpha) \in \eend(u) \subset \eend(u) + \eps$.
\end{proof}

To conclude Section~\ref{Sec_2:notation}, let us recall that the continuity of the extension $\fuz{f}:(\Fc(X),\rho)\longrightarrow(\Fc(X),\rho)$ for every metric $\rho \in \{ d_{\infty} , d_{0} , d_{S} , d_{E} \}$ is also well-known (see \cite{JardonSanSan2020_FSS_some} and \cite{Kupka2011_IS_on}). We are now ready to study the Li-Yorke chaotic-type dynamical properties of $(\Fc(X),\fuz{f})$.

\section{Li-Yorke chaos on fuzzy dynamical systems}\label{Sec_3:Li-Yorke}

After introducing Li-Yorke chaos in Subsection~\ref{SubSec_2.1:LY.MLY.DC}, in this section we extend \cite[Proposition~3]{MartinezPeRo2021_MAT_chaos} in two different directions. First, we reformulate it to include mean Li-Yorke chaos and the several variants of distributional chaos; see Theorem~\ref{The:scrambled}. Then, we extend \cite[Proposition~3]{MartinezPeRo2021_MAT_chaos} in a more sophisticated way to show that $(\Fc(X),\fuz{f})$ can exhibit chaotic properties under extremely weak hypotheses on the original dynamical system $(X,f)$ or on $(\Kc(X),\com{f})$; see Theorem~\ref{The:pairs}. Finally, in Subsection~\ref{SubSec_3.2:counter} we use Theorem~\ref{The:pairs} to provide a strong positive solution to Question~\ref{Ques:Li-Yorke.transfer}; see Examples~\ref{Exa_1:main},~\ref{Exa_2:E_MLYC}~and~\ref{Exa_3:E_DC3}.

\subsection{Two extensions of a previous result}

We start with a modest extension of \cite[Proposition~3]{MartinezPeRo2021_MAT_chaos}:

\begin{theorem}\label{The:scrambled}
	Let $f:X\longrightarrow X$ be a continuous map on a metric space $(X,d)$, and let $\eps,c,a,b,r>0$ with $0 \leq a < b$. The following statements hold:\\[-15pt]
	\begin{enumerate}[{\em(a)}]
		\item Assume that there exists $S \subset X$ fulfilling any of the next conditions: that $S$ is LY, MLY, D1~or~D2 (resp. $\eps$-)scrambled, that $S$ is \D{1} (resp.\ U-)scrambled, that $S$ is \D{2} (resp.\ $(\eps,c)$-)scrambled, or that $S$ is D3 (resp.\ $(a,b)$-)scrambled for $f$ and $d$. Then there exists a set $\com{S} \subset \Kc(X)$ with the same cardinality and fulfilling the same property as $S$ but with respect to $\com{f}$ and $d_H$.\\[-10pt]
		
		\item Assume that there is $\com{S} \subset \Kc(X)$ fulfilling any of the next conditions: that $\com{S}$ is LY, MLY, D1~or~D2 (resp. $\eps$-)scrambled, that $\com{S}$ is \D{1} (resp.\ U-)scrambled, that $\com{S}$ is \D{2} (resp.\ $(\eps,c)$-)scrambled, or that $\com{S}$ is D3 (resp.\ $(a,b)$-)scrambled for $\com{f}$ and $d_H$. Then there exists $\fuz{S} \subset \Fc(X)$ with the same cardinality and fulfilling the same property as $\com{S}$ but with respect to $\fuz{f}$ and any $\rho \in \{d_{\infty}, d_{0}, d_{S}\}$.\\[-10pt]
		
		\item Assume that there exists $\com{S} \subset \Kc(X)$ fulfilling any of the next conditions: that $\com{S}$ is LY, D1~or~D2 (resp. $\eps$-)scrambled, that $\com{S}$ is \D{1} (resp.\ U-)scrambled, or that $\com{S}$ is \D{2} (resp.\ $(\eps,c)$-)scrambled for~$\com{f}$ and~$d_H$. Then, considering the value $\eps':=\min\{ \eps , 1 \}$, there exists a set $\fuz{S} \subset \Fc(X)$ with the same cardinality as~$\com{S}$~and fulfilling, respectively, that $\fuz{S}$ is LY, D1~or~D2 (resp. $\eps'$-)scrambled, that $\fuz{S}$~is~\D{1}~(resp.\ U-)scrambled, or that $\fuz{S}$ is \D{2}~(resp.\ $(\eps',c)$-)scrambled for $\fuz{f}$ and $d_{E}$.\\[-10pt]
		
		\item Assume that there exists $\com{S} \subset \Kc(X)$ fulfilling that $\com{S}$ is MLY (resp. $\eps$-)scrambled for $\com{f}$ and $d_H$, and suppose also that the set $\{ d_H(\com{f}^j(K),\com{f}^j(L)) \ ; \ j \in \NN \}$ is bounded (resp.\ bounded by $r>0$) for every pair of sets $K,L \in \com{S}$. Then there exists a set $\fuz{S} \subset \Fc(X)$ with the same cardinality as $\com{S}$ and fulfilling that~$\fuz{S}$ is MLY (resp.\ $(\tfrac{\eps}{r+1})^2$-)scrambled with respect to $\fuz{f}$ and the metric $d_{E}$.\\[-10pt]
		
		\item Assume that there exists $\com{S} \subset \Kc(X)$ fulfilling that $\com{S}$ is D3 (resp. $(a,b)$-)scrambled for~$\com{f}$ and~$d_H$, and suppose also that for each pair of sets $K,L \in \com{S}$ with $K\neq L$ we have that $(K,L)$ is a D3~$(a_{(K,L)},b_{(K,L)})$-pair with $a_{(K,L)} < 1$ (resp.\ suppose that $a < 1$). Then there is $\fuz{S} \subset \Fc(X)$ with the same cardinality as $\com{S}$ and fulfilling that~$\fuz{S}$ is D3~(resp.\ $(a,\min\{b,1\})$-)scrambled for $\fuz{f}$ and $d_{E}$.\\[-10pt]

		\item If the dynamical system $(X,f)$ presents any of the properties of LYC, MLYC, DC1, \DC{1}, DC2, \DC{2}, DC3, U-LYC, U-MLYC, U-DC1, U-\DC{1}, U-DC2, U-\DC{2} or U-DC3, for the metric $d$, then so does the extended dynamical system $(\Kc(X),\com{f})$ for the metric $d_H$.\\[-10pt]
		
		\item If the dynamical system $(\Kc(X),\com{f})$ presents any of the properties of LYC, MLYC, DC1, \DC{1}, DC2, \DC{2}, DC3, U-LYC, U-MLYC, U-DC1, U-\DC{1}, U-DC2, U-\DC{2} or U-DC3, for the metric $d_H$, then so do the extended dynamical systems $(\Fc_{\infty}(X),\fuz{f})$, $(\Fc_{0}(X),\fuz{f})$ and $(\Fc_{S}(X),\fuz{f})$.\\[-10pt]

		\item If the dynamical system $(\Kc(X),\com{f})$ presents any of the properties of LYC, DC1, \DC{1}, DC2, \DC{2}, U-LYC, U-DC1, U-\DC{1}, U-DC2 or U-\DC{2}, for $d_H$, then so does $(\Fc_{E}(X),\fuz{f})$.\\[-10pt]
	
		\item If the dynamical system $(\Kc(X),\com{f})$ is MLYC or U-MLYC for $d_H$, and if the metric space $(X,d)$ is bounded, then the extended dynamical system $(\Fc_{E}(X),\fuz{f})$ is MLYC or U-MLYC, respectively.\\[-10pt]
		
		\item If the dynamical system $(\Kc(X),\com{f})$ is DC3 or U-DC3 for $d_H$, and if $\diam_d(X)\leq 1$, then the extended dynamical system $(\Fc_{E}(X),\fuz{f})$ is DC3 or U-DC3, respectively.
	\end{enumerate}
\end{theorem}
The notion of {\em uniform distributional chaos} used in \cite{MartinezPeRo2021_MAT_chaos} corresponds to that of U-DC1 introduced here in Subsection~\ref{SubSec_2.1:LY.MLY.DC}, so that Theorem~\ref{The:scrambled} is an extension of \cite[Proposition~3]{MartinezPeRo2021_MAT_chaos}. Moreover, we recall that the {\em $d$-diameter} of a metric space $(X,d)$ is defined as $\diam_d(X) := \sup\{ d(x,y) \ ; \ x,y \in X \}$.
\begin{proof}[Proof of Theorem~\ref{The:scrambled}]
	It is clear that statement (f) follows from statement (a), that statement (g) follows from (b), and that (h) follows from (c). Moreover, since $(X,d)$ is bounded if and only if so is the space $(\Kc(X),d_H)$, it is also easily checked that statement (i) follows from (d). In addition, statement~(j) follows from statement (e) because $\diam_d(X)\leq 1$ if and only if $\diam_{d_H}(\Kc(X))\leq 1$, and in such case every D3 $(a_{(K,L)},b_{(K,L)})$-pair $(K,L)$ for $\com{f}$ and $d_H$ must fulfill that $a_{(K,L)}<1$. Thus, to complete the proof we only need to check statements (a), (b), (c), (d) and (e).
	
	For statements (a) and (b) we can argue as in \cite[Proposition~3]{MartinezPeRo2021_MAT_chaos}, since
	\begin{align*}
		\iota_1 : X &\longrightarrow \Kc(X) \quad \text{ defined as } x \in X \mapsto \{x\} \in \Kc(X) \\[5pt]
		\iota_2 : \Kc(X) &\longrightarrow \Fc(X) \quad \text{ defined as } K \in \Kc(X) \mapsto \chi_K \in \Fc(X)
	\end{align*}
	are isometric embeddings from the metric space $(X,d)$ to $(\Kc(X),d_H)$, and from $(\Kc(X),d_H)$ to the metric spaces $\Fc_{\infty}(X)$, $\Fc_{0}(X)$ and $\Fc_{S}(X)$, respectively (see Proposition~\ref{Pro:fuzzy.metrics}). Thus, for statement (a) it is enough to consider $\com{S} := \iota_1(S)$, and for statement (b) it is enough to pick $\fuz{S} := \iota_2(\com{S})$.
	
	For statements (c), (d) and (e) one can also consider $\fuz{S} := \iota_2(\com{S})$. In fact, using Proposition~\ref{Pro:fuzzy.metrics} we have that $d_{E}(\iota_2(K),\iota_2(L)) = d_{E}(\chi_{K},\chi_{L}) = \min\{ d_H(K,L) , 1 \}$ for each pair $K,L \in \Kc(X)$, so that:
	\begin{enumerate}[(a)]
		\item[(c)] If $(K,L)$ is a LY, D1~or~a D2 $\eps$-pair, a \D{1} pair, or a \D{2} $(\eps,c)$-pair for the map~$\com{f}$ and the metric~$d_H$, setting $\eps':= \min\{\eps,1\}$ it follows that $(\chi_{K},\chi_{L})$ is a LY, D1~or~D2 $\eps'$-pair, that it is a \D{1} pair, or that it is a \D{2} $(\eps',c)$-pair for~$\fuz{f}$ and~$d_{E}$. In addition, if the set $\com{S}$ is \D{1} U-scrambled for~$\com{f}$ and the metric~$d_H$, it also trivially follows that $\iota_2(\com{S})$ is \D{1} U-scrambled for $\fuz{f}$ and $d_{E}$.
		
		\item[(d)] If $(K,L)$ is a MLY $\eps$-pair and $\{ d_H(\com{f}^j(K),\com{f}^j(L)) \ ; \ j \in \NN \}$ is bounded by $r>0$, Lemma~\ref{Lem:MLYC<->DC2} implies that $(K,L)$ is a D2 $\tfrac{\eps}{r+1}$-pair for $\com{f}$ and $d_H$, and hence  $(\chi_{K},\chi_{L})$ is a D2 $\tfrac{\eps}{r+1}$-pair for $\fuz{f}$ and $d_{E}$ by statement (c) proved above. Thus, $(\chi_{K},\chi_{L})$ is a MLY $(\tfrac{\eps}{r+1})^2$-pair for $\fuz{f}$ and $d_{E}$ by Lemma~\ref{Lem:MLYC<->DC2}.
		
		\item[(e)] If $(K,L)$ is a D3 $(a,b)$-pair for $\com{f}$ and $d_H$ with $a<1$, then $(K,L)$ is a D3 $(a,\min\{b,1\})$-pair for~$\com{f}$ and~$d_H$. It follows that $(\chi_{K},\chi_{L})$ is a D3 $(a,\min\{b,1\})$-pair for the map~$\fuz{f}$ and the metric~$d_{E}$.\qedhere
	\end{enumerate}
\end{proof}

In general, the converse of statement (f) of Theorem~\ref{The:scrambled} is false: in \cite[Theorem~10]{GarciaKwietLamOPe2009_NA_chaos} there is an example of dynamical system $(X,f)$ admitting no LY pairs and no D3 pairs but for which $(\Kc(X),d_H)$ and hence $(\Fc_{\infty}(X),\fuz{f})$, $(\Fc_{0}(X),\fuz{f})$, $(\Fc_{S}(X),\fuz{f})$ and $(\Fc_{E}(X),\fuz{f})$ are U-DC1. In Subsection~\ref{SubSec_3.2:counter} below we prove that the converse of statements (g), (h), (i) and (j) of Theorem~\ref{The:scrambled} are also false by constructing some counterexamples. These examples solve Question~\ref{Ques:Li-Yorke.transfer} in a very strong way, since Question~\ref{Ques:Li-Yorke.transfer} only asked whether the converse of statement~(g) of Theorem~\ref{The:scrambled} is false for LYC and U-DC1.

To obtain these counterexamples we are going to prove Theorem~\ref{The:pairs} below, which is an extension of \cite[Proposition~3]{MartinezPeRo2021_MAT_chaos} that is more sophisticated than Theorem~\ref{The:scrambled} above. Actually, in Theorem~\ref{The:pairs} we will prove that the existence of a unique pair of points in $(X,f)$ fulfilling any of the chaotic properties defined in Subsection~\ref{SubSec_2.1:LY.MLY.DC} is enough for $(\Fc(X),\fuz{f})$ to present an uncountable set of such pairs. However, before proving Theorem~\ref{The:pairs} we need a technical lemma regarding the metrics $d_{\infty}$, $d_{0}$, $d_{S}$, and $d_{E}$:

\begin{lemma}\label{Lem:estimations}
	Let $f:X\longrightarrow X$ be a continuous map on a metric space $(X,d)$, let $K, L \in \Kc(X)$ fulfilling that $K \subset L$ and that $K \neq L$, and set $\fuz{S} := \left\{ u^{\alpha} \ ; \ \alpha \in \ ]0,1[ \right\}$, where each $u^{\alpha} : X \longrightarrow \II$ is the function defined as
	\begin{equation}\label{eq:def.fuz{S}.u^alpha}
		u^{\alpha} := \max\{ \chi_K , \alpha\chi_L \} \quad \text{ for each } \alpha \in \ ]0,1[.
	\end{equation}
	Then, the set $\hat{S}$ is an uncountable subset of $\Fc(X)$ fulfilling that
	\begin{equation}\label{eq:d_infty}
		d_{\infty}(\fuz{f}^j(u^{\beta}),\fuz{f}^j(u^{\alpha})) = d_H(\com{f}^j(K),\com{f}^j(L)) \quad \text{ for all } j \in \NN_0,
	\end{equation}
	and that
	\begin{equation}\label{eq:rho}
		\rho(\fuz{f}^j(u^{\beta}),\fuz{f}^j(u^{\alpha})) = \min\{ d_H(\com{f}^j(K),\com{f}^j(L)) , \alpha-\beta \} \quad \text{ for all } j \in \NN_0,
	\end{equation}
	for every $0 < \beta < \alpha < 1$ and each metric $\rho \in \{ d_{0}, d_{S}, d_{E} \}$.
\end{lemma}

As previously mentioned, Lemma~\ref{Lem:estimations} is rather technical; however, this is mainly due to the need to work with the Skorokhod, sendograph, and endograph metrics, while the underlying idea of the result is quite simple. Indeed, the reader may refer to Figure~\ref{Fig:fuz_S} below, where we depict the endograph of two sets $u^{\beta}$ and $u^{\alpha}$ in $\hat{S}$, together with their images under $\fuz{f}^j$ for some $j \in \NN$.

\begin{figure}[H]
	\begin{center}
		\includegraphics[width=16.5cm]{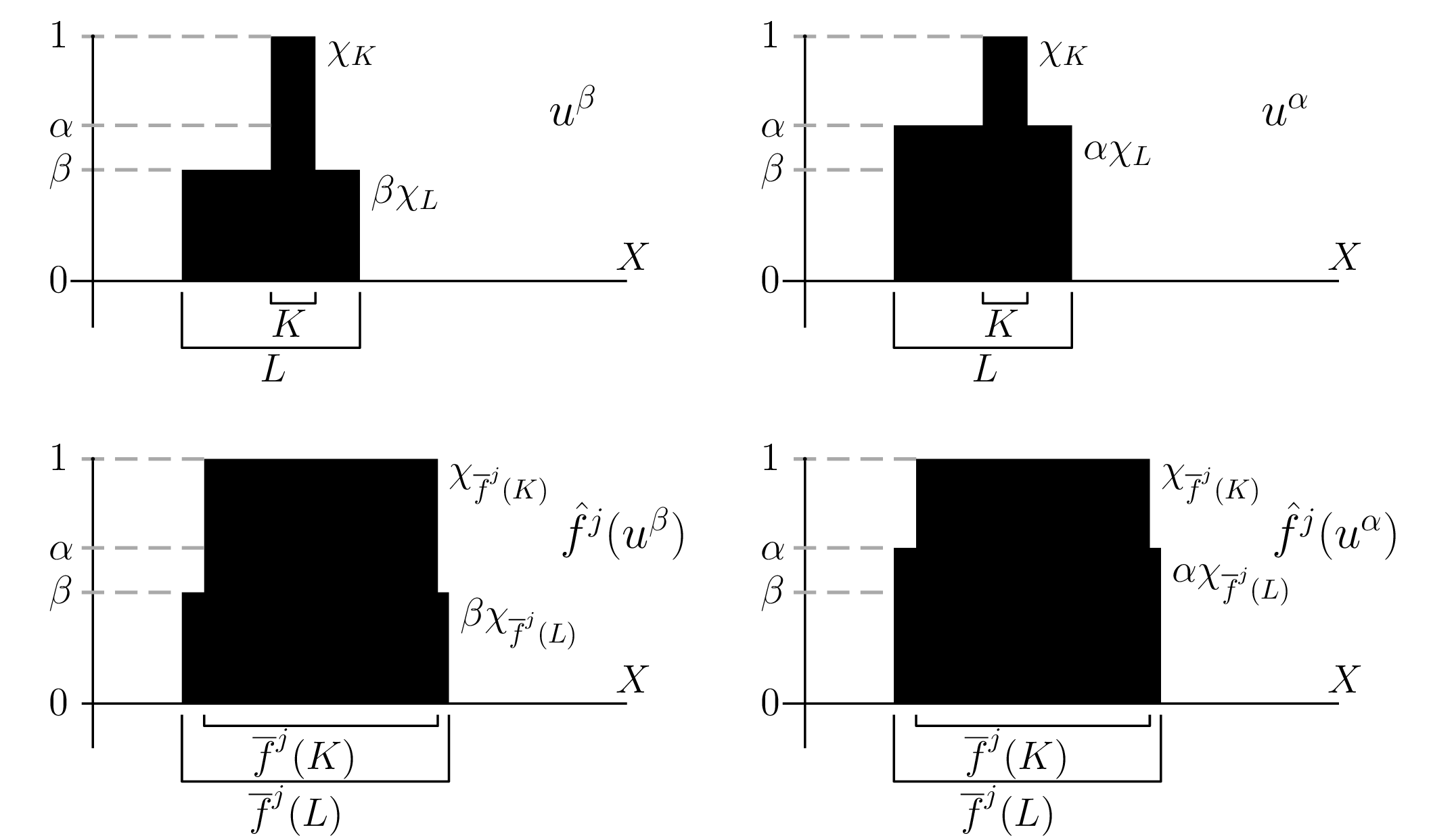}
		\caption{Representation of two elements $u^{\beta}, u^{\alpha} \in \fuz{S}$ and their $\fuz{f}^j$-iterates.}\label{Fig:fuz_S}
	\end{center}
\end{figure}

From a graphical viewpoint, it is not difficult to check that \eqref{eq:d_infty} and \eqref{eq:rho} must hold. Actually, for the representation in Figure~\ref{Fig:fuz_S} we have $\rho(u^{\beta},u^{\alpha}) = \alpha-\beta$ and $\rho(\fuz{f}^j(u^{\beta}),\fuz{f}^j(u^{\alpha})) = d_H(\com{f}^j(K),\com{f}^j(L))$ for every metric $\rho \in \{ d_{0}, d_{S}, d_{E} \}$. We now formalize this intuition:

\begin{proof}[Proof of Lemma~\ref{Lem:estimations}]
	Assume that the sets $K, L \in \Kc(X)$ fulfill that $K \subset L$ and that $K \neq L$, and define the set $\fuz{S} := \left\{ u^{\alpha} \ ; \ \alpha \in \ ]0,1[ \right\}$ as in the statement of Lemma~\ref{Lem:estimations}. Notice that, from the definition of each $u^{\alpha}$ given in \eqref{eq:def.fuz{S}.u^alpha}, it is clear that $\fuz{S}$ is an uncountable subset of $\Fc(X)$. Moreover, as represented in Figure~\ref{Fig:fuz_S}, by the definition of $\fuz{f}$ we have that
	\begin{equation}\label{eq:fuz{f}^j(u^alpha)}
		\fuz{f}^j(u^{\alpha}) = \max\left\{ \chi_{\com{f}^j(K)} , \alpha\chi_{\com{f}^j(L)} \right\} \quad \text{ for all } \alpha \in \ ]0,1[ \text{ and } j \in \NN_0.
	\end{equation}
	Using that $K \subset L$, it follows that $\com{f}^j(K) \subset \com{f}^j(L)$ for every $j \in \NN_0$ and we also have that
	\begin{equation}\label{eq:fuz{f}^j(u^alpha)_gamma}
		[\fuz{f}^j(u^{\alpha})]_{\gamma} =
		\begin{cases}
			\com{f}^j(L) & \text{ if } 0 \leq \gamma \leq \alpha, \\[5pt]
			\com{f}^j(K)  & \text{ if } \alpha < \gamma \leq 1,
		\end{cases}
		\quad \text{ for all }  \alpha \in \ ]0,1[, \ j \in \NN_0 \text{ and } \gamma \in \II.
	\end{equation}
	From \eqref{eq:fuz{f}^j(u^alpha)_gamma}, given any pair of distinct levels $0 < \beta < \alpha < 1$, any $j \in \NN_0$, and any $\gamma \in \II$, then
	\[
	d_H([\fuz{f}^j(u^{\beta})]_{\gamma},[\fuz{f}^j(u^{\alpha})]_{\gamma}) =
	\begin{cases}
		d_H(\com{f}^j(L),\com{f}^j(L)) = 0 & \text{ if } 0 \leq \gamma \leq \beta, \\[5pt]
		d_H(\com{f}^j(K),\com{f}^j(L)) & \text{ if } \beta < \gamma \leq \alpha, \\[5pt]
		d_H(\com{f}^j(K),\com{f}^j(K)) = 0 & \text{ if } \alpha < \gamma \leq 1.
	\end{cases}
	\]
	Thus, given $0 < \beta < \alpha < 1$, for the metric $d_{\infty}$ we easily obtain \eqref{eq:d_infty} since
	\[
	d_{\infty}(\fuz{f}^j(u^{\beta}),\fuz{f}^j(u^{\alpha})) = \sup_{\gamma\in\II} d_H([\fuz{f}^j(u^{\beta})]_{\gamma},[\fuz{f}^j(u^{\alpha})]_{\gamma}) = d_H(\com{f}^j(K),\com{f}^j(L)) \quad \text{ for all } j \in \NN_0.
	\]
	To prove \eqref{eq:rho}, using that $d_{E}(u,v) \leq d_{S}(u,v) \leq d_{0}(u,v) \leq d_{\infty}(u,v)$ for every pair $u,v \in \Fc(X)$ as stated in Proposition~\ref{Pro:fuzzy.metrics}, it is enough to check the following two inequalities:
	\begin{enumerate}[--]
		\item \textbf{Inequality 1}: \textit{We have that $d_{0}(\fuz{f}^j(u^{\beta}),\fuz{f}^j(u^{\alpha})) \leq \min\{ d_H(\com{f}^j(K),\com{f}^j(L)) , \alpha-\beta \}$ for all $j \in \NN_0$}. Fix any $j \in \NN_0$ and note that, by Proposition~\ref{Pro:fuzzy.metrics} and \eqref{eq:d_infty}, we already have the inequality
		\[
		d_{0}(\fuz{f}^j(u^{\beta}),\fuz{f}^j(u^{\alpha})) \leq d_{\infty}(\fuz{f}^j(u^{\beta}),\fuz{f}^j(u^{\alpha})) = d_H(\com{f}^j(K),\com{f}^j(L)).
		\]
		Now, consider the 2-piecewise-linear map $\xi_{\alpha,\beta}:\II\longrightarrow\II$ defined as
		\begin{equation*}%\label{eq:xi_apha_beta}
			\xi_{\alpha,\beta}(\gamma) :=
			\begin{cases}
				\tfrac{\beta}{\alpha}\gamma & \text{ if } 0\leq \gamma \leq \alpha,\\[5pt]
				\tfrac{1-\beta}{1-\alpha}\gamma + \tfrac{\beta-\alpha}{1-\alpha} & \text{ if } \alpha < \gamma \leq 1,
			\end{cases}
		\end{equation*}
		which fulfills that $\xi_{\alpha,\beta}(0)=0$, that $\xi_{\alpha,\beta}(\alpha)=\beta$ and that $\xi_{\alpha,\beta}(1)=1$. It is not hard to check that
		\begin{equation}\label{eq:xi_alpha_beta=alpha-beta}
			\sup_{\gamma\in\II} |\xi_{\alpha,\beta}(\gamma)-\gamma| = \left|\xi_{\alpha,\beta}\left(\alpha\right) - \alpha \right| = \left|\beta - \alpha \right| = \alpha-\beta,
		\end{equation}
		and that $\xi_{\alpha,\beta}^{-1}(\beta) = \alpha$. Thus, by \eqref{eq:fuz{f}^j(u^alpha)_gamma} we have that
		\[
		\left[ \xi_{\alpha,\beta} \circ \left( \fuz{f}^j(u^{\alpha}) \right) \right]_{\gamma} = \left[ \fuz{f}^j(u^{\alpha}) \right]_{\xi_{\alpha,\beta}^{-1}(\gamma)} =
		\begin{cases}
			\com{f}^j(L) & \text{ if } 0 \leq \gamma \leq \beta,\\[5pt]
			\com{f}^j(K) & \text{ if } \beta < \gamma \leq 1.
		\end{cases}
		\]
		This implies that $\fuz{f}^j(u^{\beta}) = \xi_{\alpha,\beta} \circ \fuz{f}^j(u^{\alpha})$, so that $d_{\infty}(\fuz{f}^j(u^{\beta}),\xi_{\alpha,\beta} \circ \fuz{f}^j(u^{\alpha}))=0$. Hence
		\[
		d_{0}(\fuz{f}^j(u^{\beta}),\fuz{f}^j(u^{\alpha})) \leq \max\left\{ \sup_{\alpha\in\II} |\xi_{\alpha,\beta}(\alpha)-\alpha| \ , \ d_{\infty}(\fuz{f}^j(u^{\beta}),\xi_{\alpha,\beta} \circ \fuz{f}^j(u^{\alpha})) \right\} \overset{\eqref{eq:xi_alpha_beta=alpha-beta}}{=} \alpha-\beta.
		\]
		Intuitively, to prove this inequality we lower the $\alpha$-level of $\fuz{f}^{j}(u^{\alpha})$ to its $\beta$-level by means of a homeomorphism $\xi_{\alpha,\beta}$ with norm equal to $\alpha-\beta$. See also Figure~\ref{Fig:fuz_S} for illustration.
		
		\item \textbf{Inequality 2}: \textit{We have that $d_{E}(\fuz{f}^j(u^{\beta}),\fuz{f}^j(u^{\alpha})) \geq \min\{ d_H(\com{f}^j(K),\com{f}^j(L)) , \alpha-\beta \}$ for all $j \in \NN_0$}. Fix any $j \in \NN_0$ and note that, if $d_H(\com{f}^j(K),\com{f}^j(L))=0$, then by Proposition~\ref{Pro:fuzzy.metrics} we would have that
		\[
		0 \leq d_{E}(\fuz{f}^j(u^{\beta}),\fuz{f}^j(u^{\alpha})) \leq d_{\infty}(\fuz{f}^j(u^{\beta}),\fuz{f}^j(u^{\alpha})) = d_H(\com{f}^j(K),\com{f}^j(L)) = 0.
		\]
		Assume now that $d_H(\com{f}^j(K),\com{f}^j(L))\neq0$ and let $0 \leq \delta < \min\{ d_H(\com{f}^j(K),\com{f}^j(L)) , \alpha-\beta \}$ be arbitrary but fixed. We will check that $\eend(\fuz{f}^j(u^{\alpha})) \not\subset \eend(\fuz{f}^j(u^{\beta})) + \delta$, so that $d_{E}(\fuz{f}^j(u^{\beta}),\fuz{f}^j(u^{\alpha}))>\delta$ by statement~(a) of Proposition~\ref{Pro:Hausdorff}. In fact, since $\com{f}^j(K) \subset \com{f}^j(L)$ but $d_H(\com{f}^j(K),\com{f}^j(L))>\delta$ there exists some $x \in \com{f}^j(L)$ for which $x \notin \com{f}^j(K)+\delta$. By \eqref{eq:fuz{f}^j(u^alpha)} we have that $(x,\alpha) \in \eend(\fuz{f}^j(u^{\alpha}))$ with
		\[
		\inf_{(y,\gamma) \in \eend(\fuz{f}^j(u^{\beta}))} \com{d}( (x,\alpha) , (y,\gamma) ) = \inf_{(y,\gamma) \in \eend(\fuz{f}^j(u^{\beta}))} \max\{ d(x,y) , |\alpha-\gamma| \},
		\]
		and for each $(y,\gamma) \in \eend(\fuz{f}^j(u^{\beta}))$, using again \eqref{eq:fuz{f}^j(u^alpha)}, we have that
		\begin{align*}
			\max\{ d(x,y) , |\alpha-\gamma| \} &= \left\{
			\begin{array}{ll}
				\max\{ d(x,y) , \alpha \} & \text{ if } \gamma=0 \text{ and } y \in X, \hspace{1cm} \\[5pt]
				\max\{ d(x,y) , |\alpha-\gamma| \} & \text{ if } 0 < \gamma \leq \beta \text{ and } y \in \com{f}^j(L), \\[5pt]
				\max\{ d(x,y) , |\alpha-\gamma| \} & \text{ if } 0 < \gamma \leq 1 \text{ and } y \in \com{f}^j(K),
			\end{array}
			\right\} \\[7.5pt]
			&\geq \left\{
			\begin{array}{ll}
				\alpha & \text{ if } \gamma=0 \text{ and } y \in X, \\[5pt]
				\alpha-\beta & \text{ if } 0 < \gamma \leq \beta \text{ and } y \in \com{f}^j(L), \\[5pt]
				d(x,y) & \text{ if } 0 < \gamma \leq 1 \text{ and } y \in \com{f}^j(K),
			\end{array}
			\right\} > \delta.
		\end{align*}
		The arbitrariness of ``$\delta$'' shows that $d_{E}(\fuz{f}^j(u^{\beta}),\fuz{f}^j(u^{\alpha})) \geq \min\{ d_H(\com{f}^j(K),\com{f}^j(L)) , \alpha-\beta \}$. Intuitively, the point $(x,\alpha)$ belongs to $\eend(\fuz{f}^j(u^{\alpha}))$, but there is no point in $\eend(\fuz{f}^j(u^{\beta}))$ closer to $(x,\alpha)$ than $(x,\beta)$. See again Figure~\ref{Fig:fuz_S} and choose $(x,\alpha)$ so that $x$ is one of endpoints of $\com{f}^j(L)$.\qedhere
	\end{enumerate}
\end{proof}

We are now ready to prove Theorem~\ref{The:pairs}, extending \cite[Proposition~3]{MartinezPeRo2021_MAT_chaos} in a very powerful way:

\begin{theorem}\label{The:pairs}
	Let $f:X\longrightarrow X$ be a continuous map acting on a metric space $(X,d)$. Hence:\\[-20pt]
	\begin{enumerate}[{\em(a)}]
		\item If there exists a LY, MLY, D1, \D{1}, D2, \D{2} or a D3 pair $(x,y)$ for~$f$ and the metric~$d$, then there exists a LY, MLY, D1, \D{1}, D2, \D{2} or a D3 pair $(K,L)$ for~$\com{f}$ and~$d_H$ such that $K \subset L$.~In particular, such a pair $(K,L)$ can be chosen fulfilling that
		\[
		d_H(\com{f}^j(K),\com{f}^j(L))=d(f^j(x),f^j(y)) \quad \text{ for all } j \in \NN_0.
		\]
		
		\item If there exists a LY, MLY, D1, \D{1}, D2, \D{2} or a D3 pair $(K,L)$ for~$\com{f}$ and~$d_H$ fulfilling $K \subset L$, then $(\Fc_{\infty}(X),\fuz{f})$ is U-LYC, U-MLYC, U-DC1, U-\DC{1}, U-DC2, U-\DC{2} or U-DC3, respectively.
		
		\item If there exists a LY, D1, \D{1}, D2 or a \D{2} pair $(K,L)$ for~$\com{f}$ and~$d_H$ fulfilling that $K \subset L$, then~$(\Fc_{0}(X),\fuz{f})$, $(\Fc_{S}(X),\fuz{f})$ and $(\Fc_{E}(X),\fuz{f})$ are LYC, DC1, \DC{1}, DC2 or \DC{2}, respectively.
		
		\item If there exists a MLY pair $(K,L)$ for~$\com{f}$ and~$d_H$ with $K \subset L$, and if $\{ d_H(\com{f}^j(K),\com{f}^j(L)) \ ; \ j \in \NN \}$ is bounded, then the systems $(\Fc_{0}(X),\fuz{f})$, $(\Fc_{S}(X),\fuz{f})$ and $(\Fc_{E}(X),\fuz{f})$ are MLYC.
	\end{enumerate}
\end{theorem}
\begin{proof}
	(a): If $(x,y)$ is a LY, MLY, D1, \D{1}, D2, \D{2} or a D3 pair for the map~$f$ and the metric~$d$, then one can consider the sets $K:=\{x\}$ and $L:=\{x,y\}$ in $\Kc(X)$. Note that
	\[
	\com{f}^j(K)=\{f^j(x)\} \quad \text{ and } \quad \com{f}^j(L)=\{f^j(x),f^j(y)\} \quad \text{ for all } j \in \NN_0,
	\]
	which easily implies that $d_H(\com{f}^j(K),\com{f}^j(L)) = d(f^j(x),f^j(y))$ for all $j \in \NN_0$. We have that $K \subset L$, and it follows that $(K,L)$ is a LY, MLY, D1, \D{1}, D2, \D{2} or a D3 pair for~$\com{f}$ and~$d_H$, respectively.
	
	To prove (b), (c), and (d), we will use Lemma~\ref{Lem:estimations}. Actually, notice that when $(K,L)$ is a LY, MLY, D1, \D{1}, D2, \D{2} or a D3 pair for~$\com{f}$ and~$d_H$, then one necessarily has that $K \neq L$. Thus:
	
	(b): Let $(K,L)$ be as in statement (b). For $\fuz{S} \subset \Fc(X)$ defined as in Lemma~\ref{Lem:estimations}, by \eqref{eq:d_infty} we get
	\[
	d_{\infty}(\fuz{f}^j(u),\fuz{f}^j(v))=d_H(\com{f}^j(K),\fuz{f}^j(L)) \quad \text{ for every pair } u,v \in \fuz{S} \text{ with } u \neq v.
	\]
	Hence, if the pair $(K,L)$ is a LY, MLY, D1 or a D2 $\eps$-pair, if it is a \D{1} pair, if it is a \D{2} $(\eps,c)$-pair, or if it is a D3 $(a,b)$-pair for~$\com{f}$ and~$d_H$, then the uncountable set $\fuz{S}$ is LY, MLY, D1 or D2 $\eps$-scrambled, \D{1}~U-scrambled, \D{2}~$(\eps,c)$-scrambled, or D3 $(a,b)$-scrambled for~$\fuz{f}$ and~$d_{\infty}$, respectively.
	
	(c): Let $(K,L)$ be as in statement (c). Pick any metric $\rho \in \{ d_{0} , d_{S} , d_{E} \}$ and then consider the uncountable set $\fuz{S} = \{ u^{\alpha} \ ; \ \alpha \in \ ]0,1[ \}$ defined as in Lemma~\ref{Lem:estimations}. Hence, if the pair $(K,L)$ is a LY, D1 or a D2 $\eps$-pair, if it is a \D{1} pair, or if it is a \D{2} $(\eps,c)$-pair for~$\com{f}$ and~$d_H$, by \eqref{eq:rho} we have that given two values $0<\beta<\alpha<1$ the pair $(u^{\beta},u^{\alpha})$ is a LY, D1 or a D2 $\min\{\eps,\alpha-\beta\}$-pair, that $(u^{\beta},u^{\alpha})$ is a \D{1} pair, or that $(u^{\beta},u^{\alpha})$ is a \D{2} $(\min\{\eps,\alpha-\beta\},c)$-pair for~$\fuz{f}$ and~$\rho$, respectively. Hence, the uncountable set $\fuz{S}$ is LY, D1, \D{1}, D2 or \D{2} scrambled for~$\fuz{f}$ and~$\rho$, respectively.
	
	(d): Let $(K,L)$ be as in statement (d). Pick any metric $\rho \in \{ d_{0} , d_{S} , d_{E} \}$ and then consider the uncountable set $\fuz{S} = \{ u^{\alpha} \ ; \ \alpha \in \ ]0,1[ \}$ defined as in Lemma~\ref{Lem:estimations}. Hence, if the pair $(K,L)$ is a MLY $\eps$-pair for~$\com{f}$ and~$d_H$, and if the set $\{ d_H(\com{f}^j(K),\com{f}^j(L)) \ ; \ j \in \NN \}$ is bounded by some positive value $r>0$, then~Lemma~\ref{Lem:MLYC<->DC2} implies that $(K,L)$ is a D2 $\tfrac{\eps}{r+1}$-pair for $\com{f}$ and $d_H$. Using this fact and arguing as in the proof of statement (c), given $0<\beta<\alpha<1$ we have that $(u^{\beta},u^{\alpha})$ is a D2 $\min\{ \tfrac{\eps}{r+1} , \alpha-\beta \}$-pair for $\fuz{f}$ and $\rho$. Applying again Lemma~\ref{Lem:MLYC<->DC2} we have that $(u^{\beta},u^{\alpha})$ is a MLYC $(\min\{ \tfrac{\eps}{r+1} , \alpha-\beta \})^2$-pair for the map~$\fuz{f}$ and the metric~$\rho$. Hence, the uncountable set $\fuz{S}$ is MLY scrambled for~$\fuz{f}$ and~$\rho$.
\end{proof}

\begin{remark}
	The statements and proof of Lemma~\ref{Lem:estimations} and Theorem~\ref{The:pairs} admit several comments:\\[-20pt]
	\begin{enumerate}[(1)]
		\item The idea behind Theorem~\ref{The:pairs} is showing how easily the Li-Yorke chaotic-type notions transfer from the systems $(X,f)$ or $(\Kc(X),\com{f})$ to $(\Fc(X),\fuz{f})$. See Subsection~\ref{SubSec_3.2:counter} for some applications.\\[-17.5pt]
		
		\item For $d_{\infty}$, $d_{S}$ and $d_{E}$ we could have considered the set $\fuz{S}$ as defined in Lemma~\ref{Lem:estimations} together with the two extra sets $u^1 := \chi_L$ and $u^0 := \chi_K$, obtaining the same estimates~\eqref{eq:d_infty}~and~\eqref{eq:rho}. However, for~$d_{0}$, although adding these two fuzzy sets to $\fuz{S}$ does not change its scrambled behaviour and the result remains valid in this case, the estimates obtained in \eqref{eq:rho} for such a metric fail.\\[-17.5pt]
		
		\item The U-LYC and U-DC2 parts of statement (b) in Theorem~\ref{The:pairs} together with the LYC and DC2 parts of statement (c) in Theorem~\ref{The:pairs} also follow when the sets of the pair $(K,L)$ are not ordered by inclusion. Indeed, let $(K,L) \in \Kc(X)\times\Kc(X)$ be any LY or D2 pair for $\com{f}$ with respect to $d_H$. Then one can verify that at least one of the pairs $(K,K\cup L)$ or $(K\cup L,L)$ is again a LY or D2 pair for $\com{f}$ with respect to $d_H$, by using the readily verified identity
		\begin{equation}\label{eq:d_H(K,L)=max{d_H(KUL,L),d_H(K,KUL)}}
			d_H(K_1,K_2) = \max\{ d_H(K_1\cup K_2,K_2) , d_H(K_1,K_1\cup K_2) \} \quad \text{ for all } K_1,K_2 \in \Kc(X).
		\end{equation}
		The conclusion then would follow from Theorem~\ref{The:pairs}. Let us explain the precise argument:
		\begin{enumerate}[--]
			\item If $(K,L)$ is a LY $\eps$-pair for some $\eps>0$, then by definition there are infinitely many integers $n \in \NN$ such that $d_H(\com{f}^{n}(K),\com{f}^{n}(L)) \geq \tfrac{\eps}{2}$. Thus, by \eqref{eq:d_H(K,L)=max{d_H(KUL,L),d_H(K,KUL)}}, we either have that
			\[
			d_H(\com{f}^{n}(K),\com{f}^{n}(K\cup L)) = d_H(\com{f}^{n}(K),\com{f}^{n}(K)\cup\com{f}^{n}(L)) = d_H(\com{f}^{n}(K),\com{f}^{n}(L)) \geq \tfrac{\eps}{2}
			\]
			or that
			\[
			d_H(\com{f}^{n}(K\cup L),\com{f}^{n}(L)) = d_H(\com{f}^{n}(K)\cup\com{f}^{n}(L),\com{f}^{n}(L)) = d_H(\com{f}^{n}(K),\com{f}^{n}(L)) \geq \tfrac{\eps}{2}
			\]
			is fulfilled infinitely many times. In both cases, since by \eqref{eq:d_H(K,L)=max{d_H(KUL,L),d_H(K,KUL)}} we also have the inequalities
			\[
			\max\{ d_H(\com{f}^j(K),\com{f}^j(K\cup L)) , d_H(\com{f}^j(K\cup L),\com{f}^j(L)) \} \leq d_H(\com{f}^j(K),\com{f}^j(L)) \quad \text{ for all } j \in \NN,
			\]
			we either have that $(K,K\cup L)$ or that $(K\cup L,L)$ is a LY $\tfrac{\eps}{2}$-pair for~$\com{f}$ and~$d_H$. From there, Theorem~\ref{The:pairs} shows that there exists an uncountable subset of $\Fc(X)$ that is LY $\tfrac{\eps}{2}$-scrambled for $(\Fc_{\infty}(X),\fuz{f})$ but also LY scrambled for $(\Fc_{0}(X),\fuz{f})$, $(\Fc_{S}(X),\fuz{f})$ and $(\Fc_{E}(X),\fuz{f})$.
			
			\item If $(K,L)$ is a D2 $\eps$-pair for some $0<\eps\leq 1$, then we use the very well-known properties of the lower and upper densities, namely that $1-\dsup(A) = \dinf(\NN \setminus A)$ for every $A \subset \NN$, and that $\dsup(A \cup B) \leq \dsup(A) + \dsup(B)$ for all $A,B \subset \NN$. Thus, by \eqref{eq:d_H(K,L)=max{d_H(KUL,L),d_H(K,KUL)}}, we obtain
			\[
			1-\dsup(\NN \setminus (A \cap B)) = \dinf(A\cap B) = \Phi_{(K,L)}(\eps) \leq 1-\eps,
			\]
			where
			\[
			\Phi_{(K,L)}(\eps) = \dinf\bigl(\{ j \in \NN \, ; \, d_H(\com{f}^j(K),\com{f}^j(L))<\eps \}\bigr),
			\]
			and the sets $A,B \subset \NN$ are defined by
			\begin{align*}
				A &:= \{ j \in \NN \, ; \, d_H(\com{f}^j(K),\com{f}^j(K\cup L))<\eps \}, \\
				B &:= \{ j \in \NN \, ; \, d_H(\com{f}^j(K\cup L),\com{f}^j(L))<\eps \}.
			\end{align*}
			Using again the above properties, it follows that
			\[
			\dsup(\NN\setminus A) + \dsup(\NN\setminus B) \geq \dsup\bigl((\NN\setminus A) \cup (\NN\setminus B)\bigr) = \dsup(\NN \setminus (A \cap B)) \geq \eps > 0,
			\]
			and hence there exists $0<\eps'\leq\eps$ such that
			\begin{equation}\label{eq:eps'}
				\dsup(\NN\setminus A) \geq \eps' \quad \text{or} \quad \dsup(\NN\setminus B) \geq \eps'.
			\end{equation}
			In particular, using once more the properties of the densities, for such $\eps'$ we have that
			\[
			\Phi_{(K,K\cup L)}(\eps) = \dinf(A) \leq 1-\eps' \quad \text{or} \quad \Phi_{(K\cup L,L)}(\eps) = \dinf(B) \leq 1-\eps',
			\]
			which implies that
			\[
			\Phi_{(K,K\cup L)}(\eps') \leq 1-\eps' \quad \text{or} \quad \Phi_{(K\cup L,L)}(\eps') \leq 1-\eps'.
			\]
			Moreover, for any $\delta>0$, it follows from \eqref{eq:d_H(K,L)=max{d_H(KUL,L),d_H(K,KUL)}} that
			\[
			1 = \Phi^*_{(K,L)}(\delta) = \dsup(A \cap B) \leq \min\{ \dsup(A) , \dsup(B) \} = \min\{ \Phi^*_{(K,K\cup L)}(\delta) , \Phi^*_{(K\cup L,L)}(\delta) \} \leq 1.
			\]
			Therefore, either $(K,K\cup L)$ or $(K\cup L,L)$ is a D2 $\eps'$-pair for $\com{f}$ with respect to $d_H$. From there, Theorem~\ref{The:pairs} shows that there exists an uncountable subset of $\Fc(X)$ that is D2 $\eps'$-scrambled for $(\Fc_{\infty}(X),\fuz{f})$, and also D2 scrambled for $(\Fc_{0}(X),\fuz{f})$, $(\Fc_{S}(X),\fuz{f})$, and $(\Fc_{E}(X),\fuz{f})$.\\[-17.5pt]
		\end{enumerate}
		
		\item The previous reasoning does not work for MLYC, DC1, \DC{1}, \DC{2}, DC3, and their uniform variants. In fact, for mean Li-Yorke chaos and these types of distributional chaos it is not enough to have infinitely many times the sets $\fuz{f}^j(K)$ and $\fuz{f}^j(K\cup L)$ or $\fuz{f}^j(K\cup L)$ and $\fuz{f}^j(L)$ far apart, nor is it enough to obtain some $\eps'>0$ satisfying the condition in \eqref{eq:eps'}; one also needs to control the frequency with which these situations occur. Actually, for MLYC one has to control the means of such distances, for DC1 one needs to achieve $\eps'=1$ in \eqref{eq:eps'}, for \DC{1} one has to ensure that $\Phi_{(K,K\cup L)}(\delta)$ or $\Phi_{(K\cup L,L)}(\delta)$ converge to $0$ as $\delta\to 0$, and for \DC{2} and DC3 one has to control how $\Phi_{(K,K\cup L)}$ or $\Phi_{(K\cup L,L)}$ evolve compared to $\Phi^*_{(K,K\cup L)}$ or $\Phi^*_{(K\cup L,L)}$ respectively.
		
		\item In Theorem~\ref{The:pairs} we have not considered DC3 for $(\Fc_{0}(X),\fuz{f})$, $(\Fc_{S}(X),\fuz{f})$ and $(\Fc_{E}(X),\fuz{f})$. This will not be a problem in our counterexamples because \DC{2} (and hence DC1) implies DC3.
	\end{enumerate}
\end{remark}

\subsection{Some counterexamples}\label{SubSec_3.2:counter}

In this subsection we present three counterexamples. With Example~\ref{Exa_1:main} we show, via Theorem~\ref{The:pairs}, that the converse of statements (g), (h), (i) and (j) in Theorem~\ref{The:scrambled} are false. This example provides a strong positive solution to Question~\ref{Ques:Li-Yorke.transfer}. In their turn, with Examples~\ref{Exa_2:E_MLYC}~and~\ref{Exa_3:E_DC3} we show that the boundedness assumptions of statements (i) and (j) in Theorem~\ref{The:scrambled} can not be weakened.

\begin{example}\label{Exa_1:main}
	\textit{There exists a continuous function $f:X\longrightarrow X$ on a metric space $(X,d)$ such that:}\\[-17.5pt]
		\begin{enumerate}[(1)]
			\item \textit{Every $S \subset X$ that is LY or D3 scrambled for $f$ and $d$ has at most two points.}
			
			\item \textit{Every $\com{S} \subset \Kc(X)$ that is LY or D3 scrambled for $\com{f}$ and $d_H$ is finite.}
			
			\item \textit{The dynamical system $(\Fc_{\infty}(X),\fuz{f})$ is U-DC1 and U-MLYC.}
			
			\item \textit{The systems $(\Fc_{0}(X),\fuz{f})$, $(\Fc_{S}(X),\fuz{f})$ and $(\Fc_{E}(X),\fuz{f})$ are DC1 and MLYC.}
		\end{enumerate}
	\begin{proof}
		Let $A \subset \NN$ fulfilling that $\dinf(A)=0$ but that $\dsup(A)=1$. These sets are well-known to exist and not difficult to be constructed. Consider $X := \NN_0\times\{0,1\}$ and let $d:X\times X\longrightarrow[0,+\infty[$ be defined for each pair $(n_1,n_2), (m_1,m_2) \in X$ as
		\[
		d( (n_1,n_2) , (m_1,m_2) ) := 
		\begin{cases}
			0 & \text{ if } (n_1,n_2)=(m_1,m_2), \\[2pt]
			|n_1-m_1| & \text{ if } n_1 \neq m_1, \\[2pt]
			1 & \text{ if } n_1 = m_1 \in \NN_0 \setminus A \text{ and } n_2 \neq m_2, \\[2pt]
			\tfrac{1}{n_1} & \text{ if } n_1 = m_1 \in A \text{ and } n_2\neq m_2.
		\end{cases}
		\]
		We leave to the reader checking that $(X,d)$ is a metric space and to verify that the topology generated by $d$ is discrete. Hence, every map on $(X,d)$ is continuous and $\Kc(X)$ is formed by the finite subsets of $X$. Consider now the map $f:X\longrightarrow X$ defined as
		\[
		f((n_1,n_2)) := (n_1+1,n_2) \quad \text{ for each } (n_1,n_2) \in X.
		\]
		To check that the dynamical system $(X,f)$ fulfills the required conditions, let us prove that $((0,0),(0,1))$ is a D1 $\eps$-pair for all $0<\eps<1$. In fact, applying $j$-times the map $f$ we have that
		\[
		d(f^j((0,0)),f^j((0,1))) = d((j,0),(j,1)) =
		\begin{cases}
			1 & \text{ if } j \in \NN\setminus A, \\[2pt]
			\tfrac{1}{j} & \text{ if } j \in A.
		\end{cases}
		\]
		An easy computation and the properties of the lower density imply, for each $0<\eps<1$, that
		\[
		\Phi_{((0,0),(0,1))}(\eps) = \dinf\left( A \ \cap \ ]\tfrac{1}{\eps},\infty[ \right) = \dinf(A) = 0.
		\]
		Moreover, the properties of the upper density implies, for all $\delta>0$, that
		\[
		\Phi^*_{((0,0),(0,1))}(\delta) = \dsup\left( A \ \cap \ ]\tfrac{1}{\delta},\infty[ \right) = \dsup(A) = 1.
		\]
		Since $((0,0),(0,1))$ is also a D2 $\eps$-pair for all $0<\eps<1$, and since $d(f^j((0,0)),f^j((0,1))) \leq 1$ for every integer $j \in \NN$, it follows from Lemma~\ref{Lem:MLYC<->DC2} that $((0,0),(0,1))$ is a MLY $\eps$-pair for all $0<\eps<1$. Using now statements (a)~and~(b) of Theorem~\ref{The:pairs} we have that $(\Fc_{\infty}(X),\fuz{f})$ is U-DC1 and U-MLYC, and using statements (a), (c) and (d) of Theorem~\ref{The:pairs} it follows that $(\Fc_{0}(X),\fuz{f})$, $(\Fc_{S}(X),\fuz{f})$ and $(\Fc_{E}(X),\fuz{f})$ are DC1 and MLYC. We have checked properties (3) and (4).
		
		To check (1) and (2) we will use the projections $P_1:X\longrightarrow\NN_0$ and $P_2:X\longrightarrow[0,1]$ on the first and second coordinates, defined as
		\begin{equation}\label{eq:P_1.P_2}
			P_1((n_1,n_2)):=n_1 \quad \text{ and } \quad P_2((n_1,n_2)):=n_2 \quad \text{ for each point } (n_1,n_2) \in X.
		\end{equation}
		Note that $P_1(f^j(x))=P_1(x)+j$ and that $P_2(f^j(x))=P_2(x)$ for all $x \in X$ and $j \in \NN_0$.	In particular, if given $K \in \Kc(X)$ and $j \in \NN_0$ we consider the set $P_1(K)+j = \{ n+j \ ; \ n \in P_1(x) \text{ for some } x \in K \}$, the equality $P_1(f^j(x))=P_1(x)+j$ implies that
		\begin{equation}\label{eq:P_1(K)}
			P_1(\com{f}^j(K)) = P_1(K)+j \quad \text{ for all } K \in \Kc(X) \text{ and } j \in \NN_0.
		\end{equation}
		In addition, for every pair of distinct points $x \neq y$ in $X$ and each $j \in \NN_0$ we have that
		\begin{equation}\label{eq:|P_1-P_1|}
			d(f^j(x),f^j(y)) =
			\begin{cases}
				|P_1(x)-P_1(y)| & \text{ if } P_1(x) \neq P_1(y), \\[2pt]
				1 & \text{ if } P_1(x) = P_1(y) \text{ and } P_1(x)+j \in \NN\setminus A, \\[2pt]
				\tfrac{1}{P_1(x)+j} & \text{ if } P_1(x) = P_1(y) \text{ and } P_1(x)+j \in A,
			\end{cases}
		\end{equation}
		where we used that $|P_1(x)-P_1(y)| = |(P_1(x)+j)-(P_1(y)+j)|$. We are ready to check (1) and (2):\\[-17pt]
		\begin{enumerate}[(1)]
			\item We claim that, if $(x,y) \in X \times X$ is a LY or D3 pair for $f$ and $d$, then $P_1(x)=P_1(y)$. In fact, if we had that $P_1(x)\neq P_1(y)$, using \eqref{eq:|P_1-P_1|} we would have that
			\[
			d(x,y) = |P_1(x)-P_2(y)| = d(f^j(x),f^j(y)) \quad \text{ for all } j \in \NN.
			\]
			This condition contradicts that $(x,y)$ is a LY or a D3 pair. As a consequence of the claim, the only possible LY and D3 pairs for $f$ and $d$ are those of the type $((n,0),(n,1)) \in X\times X$. It follows that every LY or D3 scrambled set for the map $f$ and the metric $d$ has at most two points.
			
			\item We claim that, if $(K,L) \in \Kc(X) \times \Kc(X)$ is a LY or D3 pair for $\com{f}$ and $d_H$, then $P_1(K)=P_1(L)$. In fact, if $P_1(K) \neq P_1(L)$, using \eqref{eq:P_1(K)} and \eqref{eq:|P_1-P_1|} together with the fact that $|P_1(x)-P_1(y)| \geq 1$ when $P_1(x) \neq P_1(y)$, it is not hard to check that
			\[
			d_H(K,L) = \max\left\{ |P_1(x)-P_1(y)| \ ; \ (x,y) \in K\times L \right\} = d_H(\com{f}^j(K),\com{f}^j(L)) \quad \text{ for all } j \in \NN.
			\]
			This condition contradicts that $(K,L)$ is a LY or a D3 pair. As a consequence of the claim, if a subset $\com{S} \subset \Kc(X)$ is LY or D3 scrambled for $\com{f}$ and $d_H$, and if we fix any $K \in \com{S}$, then the only possible elements of $\com{S}$ are those $L \in \Kc(X)$ fulfilling that $P_1(K)=P_1(L)$ and hence that
			\[
			L \subset \left\{ (n,m) \in X \ ; \ n \in P_1(K) \text{ and } m \in \{0,1\} \right\}.
			\]
			Since $K$ is a finite set, we have that $P_1(K)$ is finite and hence $\com{S}$ is also finite. Actually, one can even check that if $\ell \in \NN$ is the cardinal of $P_1(K)$ then the cardinal of $\com{S}$ is at most $3^{\ell}$. We deduce that every LY or D3 scrambled set for $\com{f}$ and $d_H$ is finite, as we had to check.\qedhere
		\end{enumerate}
	\end{proof}
\end{example}

Considering the dynamical system exhibited in Example~\ref{Exa_1:main} and the implications
\begin{align*}
	\text{DC1 $\Rightarrow$ \DC{1} $\Rightarrow$ DC2 $\Rightarrow$ \DC{2} $\Rightarrow$ DC3} \quad &\text{ and } \quad \text{\DC{2} $\Rightarrow$ LYC},\\[2pt]
	\text{U-DC1 $\Rightarrow$ U-\DC{1} $\Rightarrow$ U-DC2 $\Rightarrow$ U-\DC{2} $\Rightarrow$ U-DC3} \quad &\text{ and } \quad \text{U-\DC{2} $\Rightarrow$ U-LYC},
\end{align*}
we deduce that the converse of statements (g), (h), (i) and (j) in Theorem~\ref{The:scrambled} are false. In particular, with Example~\ref{Exa_1:main} we have solved Question~\ref{Ques:Li-Yorke.transfer} in the positive. Our next example shows that the boundedness assumption in statement (i) of Theorem~\ref{The:scrambled} can not be weakened:

\begin{example}\label{Exa_2:E_MLYC}
	\textit{There exists a continuous function $f:X\longrightarrow X$ on a metric space $(X,d)$ such that:}\\[-17.5pt]
	\begin{enumerate}[(1)]
		\item \textit{The systems $(X,f)$, $(\Kc(X),\com{f})$, $(\Fc_{\infty}(X),\fuz{f})$, $(\Fc_{0}(X),\fuz{f})$ and $(\Fc_{S}(X),\fuz{f})$ are U-MLYC.}
		
		\item \textit{The dynamical system $(\Fc_{E}(X),\fuz{f})$ admits no MLY pair.}
	\end{enumerate}
	\begin{proof}
		Let $A := \{ 2^{k^2} \ ; \ k \in \NN  \} \subset \NN$, consider $X := \NN_0\times[0,1]$, and let $d:X\times X\longrightarrow[0,+\infty[$ be defined for each pair $(n_1,n_2), (m_1,m_2) \in X$ as
		\[
		d( (n_1,n_2) , (m_1,m_2) ) := 
		\begin{cases}
			0 & \text{ if } (n_1,n_2)=(m_1,m_2), \\[2pt]
			|2^{n_1}-2^{m_1}| & \text{ if } n_1 \neq m_1, \\[2pt]
			\tfrac{1}{2^{n_1}} & \text{ if } n_1 = m_1 \in \NN_0 \setminus A \text{ and } n_2 \neq m_2, \\[2pt]
			n_1 & \text{ if } n_1 = m_1 \in A \text{ and } n_2\neq m_2.
		\end{cases}
		\]
		We leave to the reader checking that $(X,d)$ is a metric space. The topology generated by $d$ is discrete so that every map on $(X,d)$ is continuous and $\Kc(X)$ is formed by the finite subsets of $X$. Consider
		\[
		f((n_1,n_2)) := (n_1+1,n_2) \quad \text{ for each } (n_1,n_2) \in X.
		\]
		We claim that the uncountable set $S := \{0\}\times[0,1]$ is MLY~$1$-scrambled. Indeed, note that
		\[
		d(f^j((0,n_2)),f^j((0,m_2))) = d((j,n_2),(j,m_2)) =
		\begin{cases}
			\tfrac{1}{2^j} & \text{ if } j \in \NN\setminus A, \\[2pt]
			j & \text{ if } j \in A,
		\end{cases}
		\]
		for every pair of distinct points $(0,n_2) \neq (0,m_2)$ in $S$. It follows that
		\begin{align*}
			&\limsup_{n\to\infty} \frac{1}{n} \sum_{j=1}^n d(f^j((0,n_2)),f^j((0,m_2))) \geq \limsup_{A \ni n \to \infty} \frac{1}{n} \sum_{j=1}^n d((j,n_2),(j,m_2)) \\[5pt]
			&= \limsup_{A \ni n \to \infty} \left( \left( \sum_{j=1}^{n-1} \frac{d((j,n_2),(j,m_2))}{n} \right) + 1 \right) \geq 1.
		\end{align*}
		Moreover, if for each $n \in A$ we use the notation $[n]:=\sqrt{\log_2(n)}$, we also have that
		\begin{align*}
			0 &\leq \liminf_{n\to\infty} \frac{1}{n} \sum_{j=1}^n d(f^j((0,n_2)),f^j((0,m_2))) \leq \liminf_{A \ni n \to \infty} \frac{1}{n-1} \sum_{j=1}^{n-1} d((j,n_2),(j,m_2)) \\[5pt]
			&= \liminf_{A \ni n \to \infty} \frac{n}{n(n-1)} \left( \sum_{1 \leq j \leq n-1}^{j\in\NN\setminus A} \frac{1}{2^j} + \sum_{1\leq j\leq n-1}^{j\in A} j \right) \leq \liminf_{A \ni n \to \infty} \frac{n}{n-1} \left( \frac{1}{n} + \sum_{1\leq j\leq n-1}^{j\in A} \frac{j}{n} \right) \\[5pt]
			&= \liminf_{A \ni n \to \infty} \sum_{1\leq j\leq n-1}^{j\in A} \frac{j}{n} = \liminf_{A \ni n \to \infty} \sum_{1 \leq k < [n]} \frac{2^{k^2}}{2^{[n]^2}} = \liminf_{A \ni n \to \infty} \sum_{1 \leq k < [n]} \left(\frac{1}{2}\right)^{[n]^2-k^2} \\[5pt]
			&\leq \liminf_{A \ni n \to \infty} \left(\frac{1}{2}\right)^{[n]^2-([n]-1)^2+1} = \liminf_{A \ni n \to \infty} \left(\frac{1}{2}\right)^{2[n]} = \lim_{A \ni n \to \infty} \left(\frac{1}{2}\right)^{2[n]} = 0.
		\end{align*}
		We have obtained that $(X,f)$ is U-MLYC for $d$. Applying now statement (f) of Theorem~\ref{The:scrambled} we have that the system $(\Kc(X),\com{f})$ is U-MLYC for $d_H$, and applying statement (g) of Theorem~\ref{The:scrambled} we also obtain that the systems $(\Fc_{\infty}(X),\fuz{f})$, $(\Fc_{0}(X),\fuz{f})$ and $(\Fc_{S}(X),\fuz{f})$ are U-MLYC. We have checked (1).
		
		To check (2) we use the projections $P_1:X\longrightarrow\NN_0$ and $P_2:X\longrightarrow[0,1]$ defined as in \eqref{eq:P_1.P_2}. Since the equalities $P_1(f^j(x))=P_1(x)+j$ and $P_2(f^{j}(x))=P_2(x)$ hold for all $x \in X$ and all $j \in \NN_0$ as in~Example~\ref{Exa_1:main}, we can also use \eqref{eq:P_1(K)} here. In addition, since the map $f$ is injective we have that
		\begin{equation}\label{eq:1-to-1}
			u(x) = [\fuz{f}^j(u)](f^j(x)) \quad \text{ for all } u \in \Fc(X) \text{, } x \in X \text{ and } j \in \NN_0.
		\end{equation}
		A direct consequence of \eqref{eq:1-to-1} is that, given any $u \in \Fc(X)$ and any pair of integers $n,j \in \NN_0$, then
		\begin{equation}\label{eq:max=max.f^j}
			\max\{ u(x) \ ; \ x \in P_1^{-1}(\{n\}) \} = \max\{ [\fuz{f}^j(u)](y) \ ; \ y \in P_1^{-1}(\{n+j\}) \}.
		\end{equation}
		Note that this value can be computed because the support of $u$, namely $u_0 \in \Kc(X)$, is compact and hence finite. Let now $u\neq v$ be two distinct sets in $\Fc(X)$ and let us prove that the pair $(u,v)$ can not be a MLY~pair for the map $\fuz{f}$ and the metric $d_{E}$. We distinguish two possible cases:\\[-17.5pt]
		\begin{enumerate}[--]
			\item \textbf{Case 1}: \textit{There exists some level $\alpha \in \II$ such that $P_1(u_{\alpha}) \neq P_1(v_{\alpha})$}. In this case, renaming $u$~as~$v$ and $v$~as~$u$ if necessary, there exists some $n_1 \in P_1(v_{\alpha})\setminus P_1(u_{\alpha})$. That is, there exists $z \in X$ fulfilling that $P_1(z)=n_1$ and that $v(z) \geq \alpha$, while the value $\beta := \max\{ u(x) \ ; \ x \in P_1^{-1}(\{n_1\}) \}$ fulfills that $0 \leq \beta < \alpha$. Given any $j \in \NN$ and any value $0 \leq \delta < \alpha-\beta$ we are going to check that $\eend(\fuz{f}^j(v)) \not\subset \eend(\fuz{f}^j(u)) + \delta$. Actually, since $(z,\alpha) \in \eend(v)$ we have that the point $(f^j(z),\alpha)$ belongs to $\eend(\fuz{f}^j(v))$, and by \eqref{eq:max=max.f^j} we have that $\beta = \max\{ [\fuz{f}^j(u)](y) \ ; \ y \in P_1^{-1}(\{n_1+j\}) \}$. Since
			\[
			\inf_{(y,\gamma) \in \eend(\fuz{f}^j(u))} \com{d}( (f^j(z),\alpha) , (y,\gamma) ) = \inf_{(y,\gamma) \in \eend(\fuz{f}^j(u))} \max\{ d(f^j(z),y) , |\alpha-\gamma| \},
			\]
			and since for each $(y,\gamma) \in \eend(\fuz{f}^j(u))$ we have that
			\begin{align*}
				\max\{ d(f^j(z),y) , |\alpha-\gamma| \} &\geq \left\{
				\begin{array}{ll}
					\alpha & \text{ if } \gamma=0 \text{ and } y \in X, \hspace{1cm} \\[5pt]
					d(f^j(z),y) & \text{ if } 0 < \gamma \leq 1 \text{ and } P_1(y) \neq n_1+j, \\[5pt]
					|\alpha-\beta| & \text{ if } 0 < \gamma \leq \beta \text{ and } P_1(y) = n_1+j,
				\end{array}
				\right\} \\[7.5pt]
				&\geq \left\{
				\begin{array}{ll}
					\alpha & \text{ if } \gamma=0 \text{ and } y \in X, \hspace{1cm} \\[5pt]
					|2^{n_1+j}-2^{P_1(y)}| & \text{ if } 0 < \gamma \leq 1 \text{ and } P_1(y) \neq n_1+j, \\[5pt]
					\alpha-\beta & \text{ if } 0 < \gamma \leq \beta \text{ and } P_1(y) = n_1+j,
				\end{array}
				\right\} > \delta,
			\end{align*}
			the arbitrariness of ``$j$'' and ``$\delta$'' together with statement (a) of Proposition~\ref{Pro:Hausdorff} show that
			\[
			d_{E}(\fuz{f}^j(u),\fuz{f}^j(v)) \geq \alpha-\beta > 0 \quad \text{ for all } j \in \NN.
			\]
			It trivially follows that $(u,v)$ can not be a MLY pair, not even a LY pair, for $\fuz{f}$ and $d_{E}$.
			
			\item \textbf{Case 2}: \textit{We have that $P_1(u_{\alpha}) = P_1(v_{\alpha})$ for all $\alpha \in \II$}. In this case, given any $j \in \NN$ we claim that
			\begin{equation}\label{eq:d_E.bounded}
				\text{if } P_1(\com{f}^j(u_0)) \cap A = \varnothing, \quad \text{ then }  \quad d_{E}(\fuz{f}^j(u),\fuz{f}^j(v)) \leq \left(\frac{1}{2}\right)^{\min(P_1(\com{f}^j(u_0)))}.
			\end{equation}
			Actually, assume that $P_1(\com{f}^j(u_0)) \cap A = \varnothing$ and let $(x,\alpha)$ be any arbitrary but fixed point of the endograph $\eend(\fuz{f}^j(u))$ fulfilling that $\alpha>0$. Since $P_1(u_{\alpha})=P_1(v_{\alpha})$, using \eqref{eq:P_1(K)} it follows that
			\[
			P_1([\fuz{f}^j(u)]_{\alpha}) = P_1(\com{f}^j(u_\alpha)) = P_1(u_\alpha)+j = P_1(v_\alpha)+j = P_1(\com{f}^j(v_{\alpha})) = P_1([\fuz{f}^j(v)]_{\alpha}).
			\]
			In particular, there exists some $y \in X$ with $P_1(x)=P_1(y)$ and such that $(y,\alpha) \in \eend(\fuz{f}^j(v))$. By assumption we have that $P_1(x) = P_1(y) \in P_1(\com{f}^j(u_0)) \subset \NN \setminus A$, so that
			\[
			\com{d}( (x,\alpha) , (y,\alpha) ) = \max\{ d(x,y) , |\alpha-\alpha| \} = d(x,y) = \frac{1}{2^{P_1(x)}}.
			\]
			This implies, by the arbitrariness of $(x,\alpha) \in \eend(\fuz{f}^j(u))$ with $\alpha>0$, that
			\[
			\eend(\fuz{f}^j(u)) \subset \eend(\fuz{f}^j(v)) + \left(\frac{1}{2}\right)^{\min(P_1(\com{f}^j(u_0)))}.
			\]
			Arguing symmetrically (i.e.\ replacing $u$ with $v$ and $v$ with $u$), and by statement (a) of Proposition~\ref{Pro:Hausdorff}, we obtain \eqref{eq:d_E.bounded}. Now, note that for each $n \in \NN$ one has the inequality
			\begin{equation}\label{eq:lim}
				\frac{1}{n} \sum_{j=1}^n d_{E}(\fuz{f}^j(u),\fuz{f}^j(v)) \leq \frac{1}{n} \left( \sum_{1\leq j\leq n}^{P_1(\com{f}^j(u_0))\cap A = \varnothing} \left(\frac{1}{2}\right)^{\min(P_1(\com{f}^j(u_0)))} + \sum_{1\leq j\leq n}^{P_1(\com{f}^j(u_0))\cap A \neq \varnothing} 1 \right),
			\end{equation}
			where we used \eqref{eq:d_E.bounded} and part (a) of Proposition~\ref{Pro:fuzzy.metrics}. Looking at the first part of the last expression in~\eqref{eq:lim}, and since $\min(P_1(\com{f}^j(u_0)) = \min(P_1(u_0)+j) \geq j$ for each $j \in \NN$ by \eqref{eq:P_1(K)}, we have that
			\begin{equation}\label{eq:lim.1}
				0 \leq \lim_{n\to\infty} \frac{1}{n} \sum_{1\leq j\leq n}^{P_1(\com{f}^j(u_0))\cap A = \varnothing} \left(\frac{1}{2}\right)^{\min(P_1(\com{f}^j(u_0)))} \leq \lim_{n\to\infty} \frac{1}{n} \sum_{1\leq j\leq n} \frac{1}{2^j} \leq \lim_{n\to\infty} \frac{1}{n} = 0.
			\end{equation}
			For the second part of the last expression in \eqref{eq:lim} note that, since $u_0 \in \Kc(X)$ is a finite set, there exists some $N \in \NN$ such that $P_1(u_0) \subset [0,N]$. Thus, using \eqref{eq:P_1(K)} again we have that
			\[
			P_1(\com{f}^j(u_0)) = P_1(u_0)+j \subset [0,N]+j = [j,j+N] \quad \text{ for each } j \in \NN.
			\]
			Fix some $k \in \NN$ such that $(2^{k^2}+1) + N < 2^{(k+1)^2}$ and, for each $n>2^{k^2}$, let us denote by $\ell_n \in \NN$ the unique integer fulfilling that $n \in [2^{(k+\ell_n-1)^2}+1,2^{(k+\ell_n)^2}]$. We claim that
			\begin{equation}\label{eq:j>=2^k^2}
				\sum_{2^{k^2}< j\leq n}^{P_1(\com{f}^j(u_0))\cap A \neq \varnothing} 1 \ \leq \sum_{2^{k^2}< j\leq n}^{[j,j+N] \cap A \neq \varnothing} 1 \ = \ \#\left\{ 2^{k^2} < j \leq n \ ; \ [j,j+N] \cap A \neq \varnothing \right\} \ = \ \ell_n \cdot (N+1).
			\end{equation}
			Indeed, note that given $n \in [2^{(k+\ell_n-1)^2}+1,2^{(k+\ell_n)^2}]$ and $2^{k^2}<j\leq n$ one has that
			\[
			[j,j+N] \cap A \neq \varnothing \quad \text{ if and only if } \quad 2^{(k+m)^2}-N \leq j \leq 2^{(k+m)^2} \text{ for some } 1 \leq m \leq \ell_n.
			\]
			Thus, using \eqref{eq:j>=2^k^2} we obtain that
			\begin{align*}
				0 &\leq \lim_{n\to\infty} \frac{1}{n} \sum_{1\leq j\leq n}^{P_1(\com{f}^j(u_0))\cap A \neq \varnothing} 1 \ \leq \lim_{n\to\infty} \frac{1}{n} \left( 2^{k^2} + \sum_{2^{k^2}< j\leq n}^{[j,j+N] \cap A \neq \varnothing} 1 \right) = \lim_{n\to\infty} \frac{1}{n }\left( 2^{k^2} + \ell_n \cdot (N+1) \right) \\[5pt]
				&\leq \lim_{n\to\infty} \frac{\ell_n \cdot (N+1) + 2^{k^2}}{2^{(k+\ell_n-1)^2}} = \lim_{\ell\to\infty} \frac{\ell \cdot (N+1) + 2^{k^2}}{2^{(k+\ell-1)^2}} \leq \lim_{\ell\to\infty} \frac{\ell \cdot (N+1) + 2^{k^2}}{2^{\ell^2}} = 0.
			\end{align*}
			This computation together with \eqref{eq:lim.1} show that, taking limits in \eqref{eq:lim}, we have
			\[
			\lim_{n\to\infty} \frac{1}{n} \sum_{j=1}^n d_{E}(\fuz{f}^j(u),\fuz{f}^j(v)) = 0.
			\]
			We deduce that $(u,v)$ can not be a MLY pair for $\fuz{f}$ and $d_{E}$.\qedhere
		\end{enumerate}
	\end{proof}
\end{example}

The boundedness assumption in statement (j) of Theorem~\ref{The:scrambled} can not be weakened:

\begin{example}\label{Exa_3:E_DC3}
	\textit{There exists a continuous function $f:X\longrightarrow X$ on a metric space $(X,d)$ such that:}\\[-17.5pt]
	\begin{enumerate}[(1)]
		\item \textit{The systems $(X,f)$, $(\Kc(X),\com{f})$, $(\Fc_{\infty}(X),\fuz{f})$, $(\Fc_{0}(X),\fuz{f})$ and $(\Fc_{S}(X),\fuz{f})$ are U-DC3.}
		
		\item \textit{The dynamical system $(\Fc_{E}(X),\fuz{f})$ admits no D3 pair.}
	\end{enumerate}
	\begin{proof}
		Let $A \subset \NN$ with $\dinf(A) < \dsup(A)$. As in Example~\ref{Exa_2:E_MLYC}, consider $X := \NN_0\times[0,1]$, but this time let $d:X\times X\longrightarrow[0,+\infty[$ be defined for each pair $(n_1,n_2), (m_1,m_2) \in X$ as
		\[
		d( (n_1,n_2) , (m_1,m_2) ) := 
		\begin{cases}
			0 & \text{ if } (n_1,n_2)=(m_1,m_2), \\[2pt]
			|n_1-m_1| & \text{ if } n_1 \neq m_1, \\[2pt]
			2 & \text{ if } n_1 = m_1 \in \NN_0 \setminus A \text{ and } n_2 \neq m_2, \\[2pt]
			1 & \text{ if } n_1 = m_1 \in A \text{ and } n_2\neq m_2.
		\end{cases}
		\]
		We leave to the reader checking that $(X,d)$ is a metric space. The topology generated by $d$ is discrete so that every map on $(X,d)$ is continuous and $\Kc(X)$ is formed by the finite subsets of $X$. Consider
		\[
		f((n_1,n_2)) := (n_1+1,n_2) \quad \text{ for each } (n_1,n_2) \in X.
		\]
		We claim that the uncountable set $S := \{0\}\times[0,1]$ is D3~$(1,2)$-scrambled for $f$ and $d$. Actually, given any pair of distinct points $(0,n_2) \neq (0,m_2)$ in $S$, applying $j$-times the map $f$ we have that
		\[
		d(f^j((0,n_2)),f^j((0,m_2))) = d((j,n_2),(j,m_2)) =
		\begin{cases}
			2 & \text{ if } j \in \NN\setminus A, \\[2pt]
			1 & \text{ if } j \in A.
		\end{cases}
		\]
		Thus, $((0,n_2),(0,m_2))$ is a D3 $(1,2)$-pair for $f$ and $d$ because given any $1 < \delta < 2$ we have that
		\[
		\Phi_{((0,n_2),(0,m_2))}(\delta) = \dinf(A) < \dsup(A) = \Phi^*_{(0,n_2),(0,m_2)}.
		\]
		We have obtained that $(X,f)$ is U-DC3 for $d$. Applying now statement (f) of Theorem~\ref{The:scrambled} we have that the system $(\Kc(X),\com{f})$ is U-DC3 for $d_H$, and applying statement (g) of Theorem~\ref{The:scrambled} we also obtain that the systems $(\Fc_{\infty}(X),\fuz{f})$, $(\Fc_{0}(X),\fuz{f})$ and $(\Fc_{S}(X),\fuz{f})$ are U-DC3. We have checked (1).
		
		To check (2) we will prove that $(\Fc_{E}(X),\fuz{f})$ is an isometry. We start noticing that, as in Example~\ref{Exa_2:E_MLYC}, since the map $f$ is injective we can also use \eqref{eq:1-to-1} here. Let $u \neq v$ be two distinct normal fuzzy sets in $\Fc(X)$. Since their supports $u_0,v_0 \in \Kc(X)$ are compact and hence finite, we can compute the value $\beta := \max\left\{ |u(x)-v(x)| \ ; \ x \in X \right\}$, which is a number fulfilling that $0<\beta<1$. Renaming $u$~as~$v$ and $v$~as~$u$ if necessary, there exists some $z \in X$ such that $u(z)-v(z)=\beta$, and by \eqref{eq:1-to-1} we have that
		\[
		\beta = \max\left\{ \left| [\fuz{f}^j(u)](f^j(x))-[\fuz{f}^j(v)](f^j(x)) \right| \ ; \ x \in X \right\} \quad \text{ for all } j \in \NN_0.
		\]
		To complete the example it is enough to check the following two inequalities:\\[-17.5pt]
		\begin{enumerate}[--]
			\item \textbf{Inequality 1}: \textit{We have that $d_{E}(\fuz{f}^j(u),\fuz{f}^j(v)) \geq \beta$ for $j \in \{0,1\}$}. Fix $j \in \{0,1\}$ and consider any $0 \leq \delta < \beta$. By \eqref{eq:1-to-1} we have that $(f^j(z),u(z)) \in \eend(\fuz{f}^j(u))$ and that $[\fuz{f}^j(v)](f^j(z))=v(z)$. Thus,
			\[
			\inf_{(y,\gamma) \in \eend(\fuz{f}^j(v))} \com{d}( (f^j(z),u(z)) , (y,\gamma) ) = \inf_{(y,\gamma) \in \eend(\fuz{f}^j(v))} \max\{ d(f^j(z),y) , |u(z)-\gamma| \},
			\]
			and for each $(y,\gamma) \in \eend(\fuz{f}^j(v))$ we have that
			\begin{align*}
				\max\{ d(f^j(z),y) , |u(z)-\gamma| \} &\geq \left\{
				\begin{array}{ll}
					d(f^j(z),y) \geq 1 & \text{ if } f^j(z) \neq y, \\[5pt]
					|u(z)-v(z)| = \beta & \text{ if } f^j(z) = y,
				\end{array}
				\right\} > \delta.
			\end{align*}
			The arbitrariness of ``$\delta$'' together with statement (a) of Proposition~\ref{Pro:Hausdorff} implies \textbf{Inequality 1}.
			
			\item \textbf{Inequality 2}: \textit{We have that $d_{E}(\fuz{f}^j(u),\fuz{f}^j(v)) \leq \beta$ for $j \in \{0,1\}$}. Fix $j \in \{0,1\}$, let $(y,\alpha)$ be a point in $\eend(\fuz{f}^j(u))$, and note that $(y,[\fuz{f}^j(v)](y)) \in \eend(\fuz{f}^j(v))$. There are two possibilities:\\[-17.5pt]
			\begin{enumerate}[--]
				\item \textbf{Case 1}: \textit{We have that $[\fuz{f}^j(v)](y) \geq \alpha$}. In this case $(y,\alpha) \in \eend(\fuz{f}^j(v)) \subset \eend(\fuz{f}^j(v)) + \beta$.
								
				\item \textbf{Case 2}: \textit{We have that $[\fuz{f}^j(v)](y) < \alpha$}. In this case, since $[\fuz{f}^j(u)](y) \geq \alpha > 0$ and $f$ is injective, there exists a unique point $y_j \in X$ fulfilling that $f^j(y_j)=y$. Thus, using \eqref{eq:1-to-1} we have that
				\begin{align*}
					|\alpha - [\fuz{f}^j(v)](y)| &= \alpha - [\fuz{f}^j(v)](y) = \alpha-[\fuz{f}^j(v)](f^j(y_j)) \leq [\fuz{f}^j(u)](f^j(y_j)) - [\fuz{f}^j(v)](f^j(y_j)) \\[5pt]
					&= u(y_j)-v(y_j) \leq \max\left\{ |u(x)-v(x)| \ ; \ x \in X \right\} = \beta,
				\end{align*}
				so that $\com{d}( (y,\alpha) , (y,[\fuz{f}^j(v)](y)) ) = \max\{ d(y,y) , |\alpha-[\fuz{f}^j(v)](y)| \} \leq \beta$ and $(y,\alpha) \in \eend(\fuz{f}^j(v))+\beta$.
			\end{enumerate}
			Since $(y,\alpha) \in \eend(\fuz{f}^j(u))$ was arbitrary, it follows that
			\[
			\eend(\fuz{f}^j(u)) \subset \eend(\fuz{f}^j(v)) + \beta.
			\]
			Arguing symmetrically and using statement (a) of Proposition~\ref{Pro:Hausdorff} we obtain \textbf{Inequality 2}.
		\end{enumerate}
		We have checked that $(\Fc_{E}(X),\fuz{f})$ is an isometry. Note that an isometry admits no D3 pairs.
	\end{proof}
\end{example}

\section{Interlude: proximality and sensitivity}\label{Sec_4:sensitivity}

In Section~\ref{Sec_3:Li-Yorke} we have showed, via Example~\ref{Exa_1:main}, that Question~\ref{Ques:Li-Yorke.transfer} admits a positive answer. However, a natural question arises: under which conditions can one guarantee that a Li-Yorke chaotic-type property transfers from $(\Fc(X),\fuz{f})$ to $(\Kc(X),\com{f})$? In order to obtain a general result in this direction, and following \cite{JiangLi2025_JMAA_chaos}, we will decompose Li-Yorke chaos into two fundamental ingredients: proximality and sensitivity. In this paper, this section serves as an \textit{interlude} in which we establish some auxiliary results concerning proximality and sensitivity. These results will later be used to provide a partial negative answer to Question~\ref{Ques:Li-Yorke.transfer}, under certain natural completeness assumptions (see Section~\ref{Sec_5:Cantor}).

This section is organized as follows. In Subsection~\ref{SubSec_4.1:proximality} we discuss proximality, showing that it transfers from $(\Fc(X),\fuz{f})$ to  $(\Kc(X),\com{f})$, but that strong assumptions on $(X,f)$ or on $(\Kc(X),\com{f})$ are needed to ensure having a dense proximal relation in $(\Fc(X),\fuz{f})$; see Theorem~\ref{The:proximality}. In Subsection~\ref{SubSec_4.2:sensitivity} we focus on sensitivity, which was originally considered for $(\Fc(X),\fuz{f})$ in~\cite{JardonSan2021_IJFS_sensitivity}; see Theorem~\ref{The:sensitivity}.

Throughout this section we will use of Lemmas~\ref{Lem:key} and~\ref{Lem:key2} in a crucial way. In addition, given any dynamical system $(X,f)$ acting on a metric space $(X,d)$ and any positive integer $N \in \NN$, we will denote by $f_{(N)}:X^N\longrightarrow X^N$ the {\em $N$-fold direct product} of $f$ with itself. That is, the pair $(X^N,f_{(N)})$ will be the dynamical system
\[
f_{(N)} := \underbrace{f\times\cdots\times f}_{N} : \underbrace{X\times\cdots\times X}_{N} \longrightarrow \underbrace{X\times\cdots\times X}_{N},
\]
where $X^N:=X\times\cdots\times X$ is the {\em $N$-fold direct product} of $X$, and where the respective {\em $N$-fold direct map} is defined as $f_{(N)}\left((x_1,...,x_N)\right) := (f(x_1),...,f(x_N))$ for each $N$-tuple $(x_1,...,x_N) \in X^N$. In addition, the space $X^N$ will be endowed with the metric $d_{(N)}:X^N\times X^N\longrightarrow[0,\infty[$ defined as
\[
d_{(N)}\left( (x_1,...,x_N) , (y_1,...,y_N) \right) := \max_{1\leq l\leq N} d(x_l,y_l),
\]
for each pair of $N$-tuples $(x_1,...,x_N) , (y_1,...,y_N) \in X^N$.

\subsection{Proximality}\label{SubSec_4.1:proximality}

Given a continuous map $f:X\longrightarrow X$ acting on a metric space $(X,d)$:
\begin{enumerate}[--]
	\item A pair $(x,y) \in X\times X$ is said to be {\em proximal}, for the map~$f$ and the metric~$d$, if
	\[
	\liminf_{n\to\infty} d(f^n(x),f^n(y))=0.
	\]
	
	\item We will denote the {\em proximal relation}, for the map~$f$ and the metric~$d$, by
	\[
	\Prox(f,d) := \left\{ (x,y) \in X\times X \ ; \ (x,y) \text{ is proximal for $f$ and $d$} \right\}.
	\]
	
	\item We say that $(X,f)$ \textit{has a dense proximal relation}, for $d$, if $\Prox(f,d)$ is a dense set in $(X^2,d_{(2)})$.
\end{enumerate}
As far as we know, the property of having a dense proximal relation has not been studied in the context of dynamics on hyperspaces, so that the next result seems to be entirely new in the literature:

\begin{theorem}\label{The:proximality}
	Let $f:X\longrightarrow X$ be a continuous map acting on a metric space $(X,d)$. Hence:
	\begin{enumerate}[{\em(a)}]
		\item If the system $(X^N,f_{(N)})$ has a dense proximal relation for the metric $d_{(N)}$ for every $N \in \NN$, then so does $(\Kc(X)^N,\com{f}_{(N)})$ for the metric $(d_H)_{(N)}$ for every $N \in \NN$.
		
		\item If the system $(\Kc(X)^N,\com{f}_{(N)})$ has a dense proximal relation for the metric $(d_H)_{(N)}$ for every $N \in \NN$, then so do the systems $(\Fc_{\infty}(X),\fuz{f})$, $(\Fc_{0}(X),\fuz{f})$, $(\Fc_{S}(X),\fuz{f})$ and $(\Fc_{E}(X),\fuz{f})$.
		
		\item If any of the dynamical systems $(\Fc_{\infty}(X),\fuz{f})$, $(\Fc_{0}(X),\fuz{f})$ or $(\Fc_{S}(X),\fuz{f})$ has a dense proximal relation, then so does the system $(\Kc(X),\com{f})$ for the metric $d_H$.
	\end{enumerate}
\end{theorem}
\begin{proof}
	(a): We will use the fact: \textit{for any $K \in \Kc(X)$ and $\eps>0$ there exists some $M_0 \in \NN$ such that for each $M \geq M_0$ one can find (possibly repeated) points $x_1,...,x_M \in K$ for which every finite set of the type $K' := \{ x_1', x_2', ..., x_M' \}$ with $d(x_l,x_l')<\frac{\eps}{2}$ for all $1\leq l\leq M$ fulfills that $d_H(K,K')<\eps$}. This fact is probably folklore, but the reader can easily check it by using the compactness condition of $K$ together with statement (a) of Proposition~\ref{Pro:Hausdorff}. Fix now $N \in \NN$, let $(K_1,K_2,...,K_N) , (L_1,L_2,...,L_N) \in \Kc(X)^N$, and let $\eps>0$. To conclude we will find $(K_1',K_2',...,K_N')$ and $(L_1',L_2',...,L_N')$ in $\Kc(X)^N$ fulfilling that
	\begin{equation}\label{eq:proximal.eps-near}
		\max_{1 \leq j \leq N}\{  d_H(K_j,K_j') , d_H(L_j,L_j') \} < \eps,
	\end{equation}
	and that
	\begin{equation}\label{eq:proximal.N-tuples}
		\left( (K_1',K_2',...,K_N') , (L_1',L_2',...,L_N') \right) \in \Prox(\com{f}_{(N)},(d_H)_{(N)}).
	\end{equation}
	Using the previously stated fact, on every $K_j$ and $L_j$, we can find a positive integer $M \in \NN$ big enough to fulfill that: for each $1 \leq j \leq N$ there exist (possibly repeated) points $x_{j,1},x_{j,2},...,x_{j,M} \in K_j$ but also $y_{j,1},y_{j,2},...,y_{j,M} \in L_j$ for which every pair of finite sets of the type
	\begin{equation}\label{eq:K_j'.L_j'}
		K_j' := \{ x_{j,1}',x_{j,2}',...,x_{j,M}' \} \quad \text{ and } \quad L_j' := \{ y_{j,1}',y_{j,2}',...,y_{j,M}' \},
	\end{equation}
	with $\max\{ d(x_{j,l},x_{j,l}') , d(y_{j,l},y_{j,l}') \} < \tfrac{\eps}{2}$ for $1 \leq l \leq M$, fulfill $\max\{ d_H(K_j,K_j') , d_H(L_j,L_j') \} < \eps$. Thus, for the $NM$-tuples
	\begin{align*}
		&x := (x_{1,1},x_{1,2},...,x_{1,M},x_{2,1},x_{2,2},...,x_{2,M},...,x_{j,1},x_{j,2},...,x_{j,M},...,x_{N,1},x_{N,2},...,x_{N,M}) \in X^{NM}, \\
		&y := (y_{1,1},y_{1,2},...,y_{1,M},y_{2,1},y_{2,2},...,y_{2,M},...,y_{j,1},y_{j,2},...,y_{j,M},...,y_{N,1},y_{N,2},...,y_{N,M}) \in X^{NM},
	\end{align*}
	using that the system $(X^{NM},f_{(NM)})$ has a dense proximal relation for $d_{(NM)}$, we can find
	\begin{align*}
		&x' = (x_{1,1}',x_{1,2}',...,x_{1,M}',x_{2,1}',x_{2,2}',...,x_{2,M}',...,x_{j,1}',x_{j,2}',...,x_{j,M}',...,x_{N,1}',x_{N,2}',...,x_{N,M}') \in X^{NM}, \\
		&y' = (y_{1,1}',y_{1,2}',...,y_{1,M}',y_{2,1}',y_{2,2}',...,y_{2,M}',...,y_{j,1}',y_{j,2}',...,y_{j,M}',...,y_{N,1}',y_{N,2}',...,y_{N,M}') \in X^{NM},
	\end{align*}
	fulfilling that $d_{(NM)}(x,x')<\tfrac{\eps}{2}$, that $d_{(NM)}(y,y')<\tfrac{\eps}{2}$, and that $(x',y') \in \Prox(f_{(NM)},d_{(NM)})$. If now we choose the sets $K_j'$ and $L_j'$ as described in \eqref{eq:K_j'.L_j'}, it is routine to check that \eqref{eq:proximal.eps-near} and \eqref{eq:proximal.N-tuples} hold.
	
	(b): Since $d_{E}(u,v) \leq d_{S}(u,v) \leq d_{0}(u,v) \leq d_{\infty}(u,v)$ for every pair $u,v \in \Fc(X)$, it is enough proving that $(\Fc_{\infty}(X),\fuz{f})$ has a dense proximal relation. Let $u,v \in \Fc(X)$, fix $\eps>0$, and apply Lemma~\ref{Lem:eps.pisos} to find values $0 = \alpha_0 < \alpha_1 < \alpha_2 < ... < \alpha_N = 1$ fulfilling that $\max\{ d_H(u_0,u_{\alpha_1}) , d_H(v_0,v_{\alpha_1}) \} < \tfrac{\eps}{2}$ and that $\max\{ d_H(u_{\alpha}, u_{\alpha_{l+1}}) , d_H(v_{\alpha}, v_{\alpha_{l+1}}) \} < \tfrac{\eps}{2}$ for all $\alpha \in \ ]\alpha_l, \alpha_{l+1}]$ and $0\leq l\leq N-1$. Thus, for the $N$-tuples of compact sets $(u_{\alpha_1},u_{\alpha_2},...,u_{\alpha_N}) , (v_{\alpha_1},v_{\alpha_2},...,v_{\alpha_N}) \in \Kc(X)^N$, using our assumption that the system $(\Kc(X)^N,\com{f}_{(N)})$ has a dense proximal relation for $(d_H)_{(N)}$ we can find
	\[
	(K_1,K_2,...,K_N) , (L_1,L_2,...,L_N) \in \Kc(X)^N,
	\]
	fulfilling that $\max_{1 \leq j \leq N}\{  d_H(u_{\alpha_j},K_j) , d_H(v_{\alpha_j},L_j) \} < \tfrac{\eps}{2}$, and that
	\[
	\left( (K_1,K_2,...,K_N) , (L_1,L_2,...,L_N) \right) \in \Prox(\com{f}_{(N)},(d_H)_{(N)}).
	\]
	Considering the fuzzy sets
	\[
	u' := \max_{1\leq l\leq N} \alpha_l \chi_{K_l} \quad \text{ and } \quad v' := \max_{1\leq l\leq N} \alpha_l \chi_{L_l},
	\]
	we can check that $\max\{ d_{\infty}(u,u') , d_{\infty}(v,v') \}<\eps$. Actually, by (b) of Proposition~\ref{Pro:Hausdorff} we have that
	\[
	d_H\left(u_{\alpha_l},u'_{\alpha_l}\right) = d_H\left( \bigcup_{j\geq l} u_{\alpha_j}, \bigcup_{j\geq l} K_j \right) \leq \max_{1\leq j\leq N} \left\{ d_H\left( u_{\alpha_j} , K_j \right) \right\} < \frac{\eps}{2} \quad \text{ for all } 1\leq l\leq N,
	\]
	so that given any $\alpha \in [0,1]$ we have the bound
	\begin{align*}
		d_H\left( u_{\alpha} , u'_{\alpha} \right) &= \left\{
		\begin{array}{ll}
			d_H\left( u_{\alpha} , u'_{\alpha_1} \right), & \text{ if } \alpha \in [0,\alpha_1] \\[5pt]
			d_H\left( u_{\alpha} , u'_{\alpha_l} \right), & \text{ if } \alpha \in \ ]\alpha_l,\alpha_{l+1}], 1\leq l\leq N-1
		\end{array}
		\right\} \\[5pt]
		&\leq \left\{
		\begin{array}{ll}
			d_H\left( u_{\alpha} , u_{\alpha_1} \right) + d_H\left( u_{\alpha_1} , u_{\alpha_1}' \right), & \text{ if } \alpha \in [0,\alpha_1] \\[5pt]
			d_H\left( u_{\alpha} , u_{\alpha_{l+1}} \right) + d_H\left( u_{\alpha_{l+1}} , u_{\alpha_{l+1}}' \right), & \text{ if } \alpha \in ]\alpha_l,\alpha_{l+1}], 1\leq l\leq N-1
		\end{array}
		\right\} < \eps.
	\end{align*}
	This only shows that $d_{\infty}(u,u')<\eps$ but, symmetrically, one can also check that $d_{\infty}(v,v')<\eps$. We claim that $(u',v') \in \Prox(\fuz{f},d_{\infty})$. Indeed, again by statement (b) of Proposition~\ref{Pro:Hausdorff}, we have that
	\begin{align*}
		d_{\infty}(\fuz{f}^n(u'),\fuz{f}^n(v')) &= \sup_{\alpha\in\II} d_H(\com{f}^n(u'_{\alpha}),\com{f}^n(v'_{\alpha})) = \max_{1 \leq l \leq N} d_H\left(\com{f}^n\left(\bigcup_{j \geq l} K_j\right),\com{f}^n\left(\bigcup_{j \geq l} L_j\right)\right) \\
		&= \max_{1 \leq l \leq N} d_H\left(\bigcup_{j \geq l} \com{f}^n(K_j),\bigcup_{j \geq l} \com{f}^n(L_j)\right) \leq \max_{1 \leq j \leq N} d_H\left( \com{f}^n(K_j),\com{f}^n(L_j) \right).
	\end{align*}
	The last expression equals $(d_H)_{(N)}\left( \com{f}_{(N)}^n(K_1,K_2,...,K_N) , \com{f}_{(N)}^n(L_1,L_2,...,L_N) \right)$ and the result follows.
	
	(c): Fix $\rho \in \{ d_{\infty} , d_{0} , d_{S} \}$ and assume that $\Prox(\fuz{f},\rho)$ is dense in $(\Fc(X)^2,\rho_{(2)})$. We must prove that the set $\Prox(\com{f},d_H)$ is dense in $(\Kc(X)^2,(d_H)_{(2)})$. To do so let $K,L \in \Kc(X)$ and $\eps>0$ be arbitrary but fixed, and let us find $K',L' \in \Kc(X)$ fulfilling the properties
	\begin{equation}\label{eq:proximal}
		\max\{ d_H(K,K') , d_H(L,L') \} < \eps \quad \text{ and } \quad (K',L') \text{ is proximal for $\com{f}$ and $d_H$}.
	\end{equation}
	Indeed, by assumption there exists a pair $(u,v) \in \Fc(X)\times\Fc(X)$ fulfilling that
	\[
	(u,v) \in \Bc_{\rho}(\chi_{K},\eps)\times\Bc_{\rho}(\chi_{L},\eps) \cap \Prox(\fuz{f},\rho).
	\]
	Let $K':=u_0$ and $L':=v_0$. Since $d_H(w_0,w_0') \leq \min\{ d_{S}(w,w') , d_{0}(w,w') , d_{\infty}(w,w') \} \leq \rho(w,w')$ for every pair of normal fuzzy sets $w,w' \in \Fc(X)$ by Proposition~\ref{Pro:fuzzy.metrics}, we have that
	\[
	\max\{ d_H(K,K') , d_H(L,L') \} \leq \max\{ \rho(\chi_{K},u) , \rho(\chi_{L},v) \} < \eps,
	\]
	but also that
	\[
	d_H\left(\com{f}^n(K'),\com{f}^n(L')\right) = d_H\left(\com{f}^n(u_0),\com{f}^n(v_0)\right) \leq \rho\left(\fuz{f}^n(u),\fuz{f}^n(v)\right) \quad \text{ for every } n \in \NN.
	\]
	As $(u,v) \in \Prox(\fuz{f},\rho)$, it follows that $(K',L') \in \Prox(\com{f},d_H)$ and \eqref{eq:proximal} is fulfilled.
\end{proof}

\subsection{Sensitivity and strong sensitivity}\label{SubSec_4.2:sensitivity}

Following \cite{AbrahamBiauCa2002_JMAA_chaotic,JardonSan2021_IJFS_sensitivity,SharmaNa2010_TA_inducing,WangWeiCamp2009_TA_sensitive}, given a continuous map $f:X\longrightarrow X$ acting on a metric space $(X,d)$:
\begin{enumerate}[--]
	\item We say that the system $(X,f)$ is {\em sensitive} (usually called {\em sensitive on initial conditions}), for the metric~$d$, if there exists some constant $\eps>0$ (called {\em sensitivity constant}) such that for each $x \in X$ and $\delta>0$ there are $y \in X$ and $n \in \NN$ fulfilling that $d(x,y)<\delta$ and that $d(f^n(x),f^n(y))>\eps$.
	
	\item We say that the dynamical system $(X,f)$ is {\em collectively sensitive}, for the metric~$d$, if there exists some constant $\eps>0$ (called {\em collective sensitivity constant}) such that given any finite sequence of distinct points $x_1,x_2,...,x_N \in X$ and any $\delta>0$ there are $y_1,y_2,...,y_N \in X$ and $n \in \NN$ for which:
	\begin{enumerate}[(CS1)]
		\item we have that $d(x_j,y_j)<\delta$ for all $1 \leq j \leq N$;
		
		\item there exists some $j_0 \in \{1,2,...,N\}$ such that either $d(f^n(x_{j_0}),f^n(y_j))>\eps$ for all $1 \leq j \leq N$ or $d(f^n(x_j),f^n(y_{j_0}))>\eps$ for all $1 \leq j \leq N$.
	\end{enumerate}
	
	\item We say that the dynamical system $(X,f)$ is {\em strongly sensitive}, for the metric $d$, if there exists some constant $\eps>0$ (called {\em strong sensitivity constant}) such that for each $x \in X$ and $\delta>0$ there is $n_0 \in \NN$ fulfilling that given any $n \geq n_0$ we can find $y_n \in X$ for which $d(x,y_n)<\delta$ and $d(f^n(x),f^n(y_n))>\eps$.
\end{enumerate}
The notion of sensitivity was included in the original definition of Devaney chaos and it is sometimes referred to as the ``butterfly effect'', meaning that small changes in initial conditions may produce large variations after some time. The concept of strong sensitivity was introduced in \cite{AbrahamBiauCa2002_JMAA_chaotic} and has been called \textit{cofinite sensitivity} in other references (see \cite{YaoZhu2023_MAT_syndetic}). In its turn, collective sensitivity was considered for the first time in \cite{WangWeiCamp2009_TA_sensitive}. The reader should note that we have the implications
\[
\textit{strong sensitivity $\Rightarrow$ collective sensitivity $\Rightarrow$ sensitivity}.
\]
Actually, that collective sensitivity implies sensitivity is direct, while the fact that the ``strong'' version implies the ``collective'' one can be deduced from \cite[Theorem~2.3]{WangWeiCamp2009_TA_sensitive} and \cite[Proposition~2.8]{SharmaNa2010_TA_inducing}, or from Theorem~\ref{The:sensitivity} below. The converse implications are false (see \cite[Example~2.11]{SharmaNa2010_TA_inducing} and \cite[Example~1]{YaoZhu2023_MAT_syndetic}).

In the context of dynamics on hyperspaces, these properties were studied for the system $(\Kc(X),\com{f})$ in \cite{SharmaNa2010_TA_inducing,WangWeiCamp2009_TA_sensitive}, and later for the system $(\Fc(X),\fuz{f})$ in \cite{JardonSan2021_IJFS_sensitivity}. The next theorem was essentially proved in such references, if one combines properly their results. However, since one of the proofs given in \cite{JardonSan2021_IJFS_sensitivity} is not completely clear, and because we include an additional result concerning the endograph metric, we find it convenient to restate the theorem in full and to re-prove the unclear part:

\begin{theorem}\label{The:sensitivity}
	Let $f:X\longrightarrow X$ be a continuous map acting on a metric space $(X,d)$. Hence:
	\begin{enumerate}[{\em(a)}]
		\item The following statements are equivalent:
		\begin{enumerate}[{\em(i)}]
			\item the system $(X,f)$ is collectively sensitive for the metric $d$;
			
			\item the system $(\Kc(X),\com{f})$ is sensitive for the metric $d_H$;
			
			\item the system $(\Fc_{\infty}(X),\fuz{f})$ is sensitive;
			
			\item the system $(\Fc_{0}(X),\fuz{f})$ is sensitive;
			
			\item the system $(\Fc_{S}(X),\fuz{f})$ is sensitive.
		\end{enumerate}
		
		\item If the system $(\Kc(X),\com{f})$ is (resp.\ strongly) sensitive for $d_H$, then so is $(X,f)$ for the metric $d$.
		
		\item The following statements are equivalent:
		\begin{enumerate}[{\em(i)}]
			\item the system $(X,f)$ is strongly sensitive for the metric $d$;
			
			\item the system $(\Kc(X),\com{f})$ is strongly sensitive for the metric $d_H$;
			
			\item the system $(\Fc_{\infty}(X),\fuz{f})$ is strongly sensitive;
			
			\item the system $(\Fc_{0}(X),\fuz{f})$ is strongly sensitive;
			
			\item the system $(\Fc_{S}(X),\fuz{f})$ is strongly sensitive.
		\end{enumerate}
		
		\item If the system $(\Fc_{E}(X),\fuz{f})$ is (resp.\ strongly) sensitive, then so is $(\Kc(X),\com{f})$ for the metric $d_H$.
	\end{enumerate}
\end{theorem}
\begin{proof}
	(a): The equivalence (i) $\Leftrightarrow$ (ii) is \cite[Theorem~2.3]{WangWeiCamp2009_TA_sensitive}. The implications (ii) $\Rightarrow$ (iii), (iv), (v) were proved in \cite[Theorem~3.10]{JardonSan2021_IJFS_sensitivity} with a correct proof. However, for the implications (iii), (iv), (v) $\Rightarrow$ (ii), the authors in \cite[Theorem~3.11]{JardonSan2021_IJFS_sensitivity} stated that such implications are a consequence of the equalities
	\[
	d_{S}(\chi_{\{x\}},u) = d_{0}(\chi_{\{x\}},u) = d_{\infty}(\chi_{\{x\}},u) = d_H(\{x\},u_0) \quad \text{ for all } x \in X \text{ and } u \in \Fc(X).
	\]
	Although these equalities are true (see Proposition~\ref{Pro:fuzzy.metrics}), the fact that (iii), (iv), (v) $\Rightarrow$ (ii) follow from them is not clear. Indeed, recall that for the system $(\Kc(X),\com{f})$ to be sensitive, it is not enough that the system $(X,f)$ is sensitive (see \cite[Example~2.11]{SharmaNa2010_TA_inducing}). Thus, we consider convenient to include a proof of the following \textit{fact} (which uses both Lemmas~\ref{Lem:key}~and~\ref{Lem:key2} in an essential way):
	\begin{enumerate}[--]
		\item \textit{If any of the extended dynamical systems $(\Fc_{\infty}(X),\fuz{f})$, $(\Fc_{0}(X),\fuz{f})$, $(\Fc_{S}(X),\fuz{f})$ or $(\Fc_{E}(X),\fuz{f})$ is sensitive, then so is $(\Kc(X),\com{f})$ for the metric $d_H$.}
	\end{enumerate}
	\textit{Proof of the fact}: Fix $\rho \in \{ d_{\infty} , d_{0} , d_{S} , d_{E} \}$, assume that $(\Fc(X),\fuz{f})$ is sensitive for $\rho$, and let $\eps>0$ be a sensitivity constant, i.e.\ we are assuming that for any $u \in \Fc(X)$ and any $\delta>0$ there are $v \in \Fc(X)$ and $n \in \NN$ fulfilling that $\rho(u,v)<\delta$ and that $\rho(\fuz{f}^n(u),\fuz{f}^n(v))>\eps$. Without loss of generality assume that $0<\eps<\tfrac{1}{2}$. We claim that $\eps$ is a sensitivity constant for $(\Kc(X),\com{f})$ and $d_H$. Indeed, given any compact set $K \in \Kc(X)$ and any positive value $0<\delta<\eps$, if we set $u:= \chi_K$, by assumption there exist a normal fuzzy set $v \in \Fc(X)$ and a positive integer $n \in \NN$ fulfilling that
	\[
	\rho(\chi_K,v)<\delta \quad \text{ and } \quad \rho(\chi_{\com{f}^n(K)},\fuz{f}^n(v)) = \rho(\fuz{f}^n(\chi_K),\fuz{f}^n(v))>\eps.
	\]
	Recalling that $\com{f}^n(v_{\alpha})=f^n(v_{\alpha})=[\fuz{f}^n(v)]_{\alpha}$ for every $\alpha \in \II$, we have three cases:
	\begin{enumerate}[--]
		\item \textbf{Case 1}: \textit{We have that $\rho \in \{d_{\infty},d_{0}\}$}. In this case, by statement (c) of Proposition~\ref{Pro:fuzzy.metrics} we have that $d_{0}(\chi_K,v) = d_{\infty}(\chi_K,v)<\delta$ and that $d_{0}(\chi_{\com{f}^n(K)},\fuz{f}^n(v)) = d_{\infty}(\chi_{\com{f}^n(K)},\fuz{f}^n(v))>\eps$. It follows that
		\[
		d_H(\com{f}^n(K),\com{f}^n(v_{\alpha}))>\eps \quad \text{ for some } \alpha \in \II,
		\]
		and that $d_H(K,v_{\alpha})<\delta$ for such $\alpha \in \II$ by the definition of $d_{\infty}$. The result follows since $v_{\alpha} \in \Kc(X)$.
					
		\item \textbf{Case 2}: \textit{We have that $\rho = d_{S}$}. In this case, part (b) of Lemma~\ref{Lem:key2} shows that there exists some level $\alpha \in [0,1-\eps] \subset [0,1-\delta]$ such that $d_H(\com{f}^n(K),\com{f}^n(v_{\alpha}))>\eps$. In its turn, part (a) of Lemma~\ref{Lem:key2} shows that $d_H(K,v_{\alpha})<\delta$. The result follows since $v_{\alpha} \in \Kc(X)$.
			
		\item \textbf{Case 3}: \textit{We have that $\rho = d_{E}$}. In this case, part (c) of Lemma~\ref{Lem:key2} shows that there exists some level $\alpha \in \ ]\eps,1-\eps] \subset \ ]\delta,1-\delta]$ such that $d_H(\com{f}^n(K),\com{f}^n(v_{\alpha}))>\eps$. An application of Lemma~\ref{Lem:key} shows now that $d_H(K,v_{\alpha})<\delta$. The result follows since $v_{\alpha} \in \Kc(X)$.
	\end{enumerate}
	We deduce that $(\Kc(X),\com{f})$ admits $\eps$ as a sensitivity constant for the metric $d_H$, as required. With these arguments, the equivalences stated in (a) are now proved in full generality.
	
	(b): A proof for this statement can be found in \cite[Propositions~2.1~and~2.8]{SharmaNa2010_TA_inducing}.
	
	(c): The equivalence (i) $\Leftrightarrow$ (ii) is \cite[Propositions~2.3~and~2.8]{SharmaNa2010_TA_inducing}. That (i) $\Rightarrow$ (iii), (iv), (v) was proved in \cite[Theorem~3.7]{JardonSan2021_IJFS_sensitivity}, and the implications (iii), (iv), (v) $\Rightarrow$ (i) were proved in \cite[Proposition~3.3]{JardonSan2021_IJFS_sensitivity}.
	
	(d): We have already proved the sensitivity part in statement (a). For strong sensitivity one can fix a general integer $n_0 \in \NN$, and repeat the arguments above (see \textbf{Case 3}) for each $n \geq n_0$.
\end{proof}

With Theorems~\ref{The:proximality} and~\ref{The:sensitivity} at hand, we are now ready to proceed to Section~\ref{Sec_5:Cantor}.

\section{Cantor-dense Li-Yorke chaos}\label{Sec_5:Cantor}

In this section we show that, although Question~\ref{Ques:Li-Yorke.transfer} admits a positive answer by Example~\ref{Exa_1:main}, under certain conditions on the system $(X,f)$ one can guarantee that some Li-Yorke chaotic-type properties transfer from $(\Fc(X),\fuz{f})$ to $(\Kc(X),\com{f})$. In particular, following the theory from \cite{JiangLi2025_JMAA_chaos}, we introduce the concept of {\em Cantor-dense Li-Yorke chaos} and we prove that it transfers from $(\Fc(X),\fuz{f})$ to $(\Kc(X),\com{f})$ whenever $(X,d)$ is a complete metric space; see Theorem~\ref{The:cantor}. By adding further assumptions, such as the linearity of the underlying dynamical system, we even obtain certain equivalences with the original system; see Theorem~\ref{The:operators}. Throughout this section, Theorems~\ref{The:proximality} and~\ref{The:sensitivity} will play a crucial role.

\subsection{Another strengthened variant of Li-Yorke chaos}

We will use a dynamical notion stronger than U-LYC but that has been recently considered in the context of dynamical systems on completely metrizable spaces (see \cite{JiangLi2025_JMAA_chaos}). From now on, a {\em Cantor set} will be a set homeomorphic to the classical Cantor ternary set of the unit interval.

\begin{definition}\label{Def:Cantor-dense}
	Let $f:X\longrightarrow X$ be a continuous map acting on a metric space $(X,d)$. We say that the dynamical system $(X,f)$ is {\em Cantor-densely Li-Yorke chaotic}, for the metric $d$, if there exists some positive value $\eps>0$ such that for any sequence $(U_j)_{j\in\NN}$ of non-empty open subsets of $X$ there is a sequence of Cantor sets $(K_j)_{j\in\NN}$ in $X$ fulfilling the properties
	\[
	K_j \subset U_j \text{ for all } j \in \NN \quad \text{ and } \quad \bigcup_{j\in\NN} K_j \text{ is a Li-Yorke $\eps$-scrambled set for $f$~and~$d$.}
	\]
	For short, we will write that $(X,f)$ is {\em CD-LYC}.
\end{definition}

Since every Cantor set is uncountable, it is clear that CD-LYC $\Rightarrow$ U-LYC $\Rightarrow$ LYC. Moreover, it has been recently proved, using the celebrated {\em Mycielski theorem} (see \cite[Theorem~2.2]{JiangLi2025_JMAA_chaos}), that
\begin{enumerate}[--]
	\item \cite[Proposition~3.5]{JiangLi2025_JMAA_chaos}: \textit{Given a continuous map $f:X\longrightarrow X$ acting on a \textbf{completely metrizable} space $(X,d)$, then the dynamical system $(X,f)$ is CD-LYC for the metric $d$ if and only if $(X,f)$ has a dense proximal relation and it is a sensitive dynamical system for the metric $d$}.
\end{enumerate}
Using the results obtained in Section~\ref{Sec_4:sensitivity} on proximality and sensitivity we obtain that: 

\begin{theorem}\label{The:cantor}
	Let $f:X\longrightarrow X$ be a continuous map on a complete metric space $(X,d)$. If any of the systems $(\Fc_{\infty}(X),\fuz{f})$, $(\Fc_{0}(X),\fuz{f})$ or $(\Fc_{S}(X),\fuz{f})$ is CD-LYC, then so is $(\Kc(X),\com{f})$ for the metric $d_H$.
\end{theorem}
\begin{proof}
	It is almost trivial to check that, if a dynamical system is CD-LYC, then it has a dense proximal relation and it is sensitive. Thus, the hypothesis of this result together with statement~(c) of~Theorem~\ref{The:proximality} and statement (a) of~Theorem~\ref{The:sensitivity} imply that the system $(\Kc(X),\com{f})$ has a dense proximal relation and that it is sensitive for the metric $d_H$. Finally, since $(\Kc(X),d_H)$ is complete when $(X,d)$ is complete (see for instance \cite[Exercise~2.15]{IllanesNad1999_book_hyperspaces}), the result follows from~\cite[Proposition~3.5]{JiangLi2025_JMAA_chaos}.
\end{proof}

From Theorem~\ref{The:cantor} we deduce that Question~\ref{Ques:Li-Yorke.transfer} may have a general negative answer for certain strengthened versions of LYC. Let us show that similar arguments apply to linear dynamical systems.

\subsection{The case of infinite-dimensional linear operators}

As a subfield of Topological Dynamics one finds the so-called {\em Linear Dynamics}, where the dynamical systems under consideration are usually denoted by $(X,T)$ and consist of a {\em Fr\'echet space} $(X,d)$ and a {\em continuous linear operator}, or just {\em operator} for short, $T:X\longrightarrow X$. Recall that a Fr\'echet space is a locally convex topological vector space $X$, admitting a translation-invariant metric $d$ for which $(X,d)$ is complete. Throughout this section, given a Fr\'echet space $(X,d)$ we denote by~$0_X$ the zero-vector in~$X$, and given any subset $E \subset X$ we denote by $\lspan(E)$ its linear span. In addition, given any operator $T:X\longrightarrow X$ and any vector $x \in X$ we will denote its {\em $T$-orbit} by $\orb_T(x) := \{ T^n(x) \ ; \ n \in \NN_0 \}$, and we will denote the {\em set of vectors with a Null Sub-$T$-orbit} by
\[
NS(T) := \{ x \in X \ ; \ (T^n(x))_n \text{ admits a $0_X$-convergent subsequence} \}.
\]
The linear system $(X,T)$ is called {\em equicontinuous} if for each $\eps>0$ there exists some $\delta>0$ fulfilling that $d(x,y)<\delta$ implies $d(T^n(x),T^n(y))<\eps$ for all $n \in \NN$. Since the metric $d$ is translation-invariant, the previous condition is equivalent to the fact that for each $\eps>0$ there exists some $\delta>0$ fulfilling that $d(x,0_X)<\delta$ implies $d(T^n(x),0_X)<\eps$ for all $n \in \NN$. In Functional Analysis, this property is usually called {\em power-boundedness}, but we will keep the word ``{\em equicontinuous}'' for the rest of the paper.

When the Fr\'echet space $X$ is finite-dimensional, then the dynamics of any operator $T:X\longrightarrow X$ are completely determined by the Jordan canonical form of $T$ and, in particular, Li-Yorke chaos can not occur. However, when $X$ is infinite-dimensional the situation becomes substantially richer, and there exist several works (such as \cite{BernardesBoMuPe2015_ETDS_Li-Yorke,JiangLi2025_JMAA_chaos,Lopez2025_MJOM_frequently}) in which chaos has been studied for infinite-dimensional linear systems. In the context of dynamics on hyperspaces, various Li-Yorke chaotic-type properties for the respective extensions $(\Kc(X),\com{T})$ and $(\Fc(X),\fuz{T})$ were explored in \cite{BernardesPeRo2017_IEOT_set-valued,MartinezPeRo2021_MAT_chaos}. In this subsection we extend \cite[Theorem~3.2]{BernardesPeRo2017_IEOT_set-valued} and \cite[Theorem~3]{MartinezPeRo2021_MAT_chaos} by considering the metrics $d_{S}$~and~$d_{E}$, but also including the notion of CD-LYC in the obtained equivalences; see Theorem~\ref{The:operators}. To do so we strongly rely on the theory developed so far in this paper, together with the next two auxiliary (possibly folklore) lemmas:

\begin{lemma}\label{Lem:equicontinuous}
	Let $T:X\longrightarrow X$ be an operator on a Fr\'echet space $(X,d)$. Assume that $\lspan(NS(T))$ is dense in $X$ and that $(X,T)$ is equicontinuous. Then, for all $K \in \Kc(X)$ and all $u \in \Fc(X)$ we have
	\[
	\lim_{n\to\infty} d_H\left(\com{T}^n(K),\{0_X\}\right) = 0 \quad \text{ and } \quad \lim_{n\to\infty} d_{\infty}\left(\fuz{T}^n(u),\chi_{\{0_X\}}\right) = 0.
	\]
\end{lemma}
\begin{proof}
	When $(X,T)$ is equicontinuous, using the arguments in \cite[Lemma~3.2]{JiangLi2025_JMAA_chaos} it almost trivial to check that every orbit with a $0_X$-convergent subsequence is indeed full $0_X$-convergent, that is
	\[
	NS(T) = \left\{ x \in X \ ; \ \lim_{n\to\infty} d(T^n(x),0_X) = 0 \right\}.
	\]
	Such a set is a linear subspace, so that $NS(T)=\lspan(NS(T))$ and $NS(T)$ is dense in $X$. We claim that $NS(T)=X$: actually, note that given any $x \in X$ and any $y \in NS(T)$ we have that
	\[
	0 \leq \limsup_{n\to\infty} d(T^n(x),0_X) \leq \limsup_{n\to\infty} \left( d(T^n(x),T^n(y)) + d(T^n(y),0_X) \right) = \limsup_{n\to\infty} d(T^n(x),T^n(y)),
	\]
	which is arbitrarily small by the equicontinuity of $(X,T)$ and the density of $NS(T)$. Now, let us fix any $K \in \Kc(X)$ and let $\eps>0$. Using that $(X,T)$ is equicontinuous, that $K \subset NS(T)$, and that $K$~is compact, it is not hard to find $n_{K,\eps} \in \NN$ such that $d(T^n(x),0_X) < \eps$ for all $x \in K$ and $n \geq n_{K,\eps}$.~It follows that $d_H(\com{T}^n(K),\{0_X\})<\eps$ for all $n \geq n_{K,\eps}$, so that $d_H(\com{T}^n(K),\{0_X\})$ converges to $0$. In its turn, given any $u \in \Fc(X)$ we have that $u_0 \in \Kc(X)$. Therefore, repeating the argument for the compact set $K=u_0$ together with statement (b) of Proposition~\ref{Pro:fuzzy.metrics} completes the proof.
\end{proof}

\begin{lemma}\label{Lem:collective}
	Let $T:X\longrightarrow X$ be an operator acting on a Fr\'echet space $(X,d)$. Then, the following statements are equivalent:
	\begin{enumerate}[{\em(i)}]
		\item the system $(X,T)$ is not equicontinuous;
		
		\item the system $(X,T)$ is sensitive for the metric $d$;
		
		\item the system $(X,T)$ is collectively sensitive for the metric $d$. 
	\end{enumerate}
\end{lemma}
\begin{proof}
	It is well-known that (i)~$\Leftrightarrow$~(ii); see for instance \cite[Theorem~3.11]{JiangLi2025_JMAA_chaos}. In addition, by the definition of sensitivity and collective sensitivity, it is clear that (iii) $\Rightarrow$ (ii). To complete the proof we only need to check that (ii) $\Rightarrow$ (iii). Assume (ii), let $\eps>0$ be a sensitivity constant for~$(X,T)$, and let us check that $\tfrac{\eps}{2}$ is a collective sensitivity constant for $(X,T)$. Actually, given any finite sequence of distinct points $x_1,x_2,...,x_N \in X$ and any $\delta>0$, by the sensitivity of $(X,T)$ we can find $\tilde{x} \in X$ and some $n \in \NN$ fulfilling that $d(\tilde{x},0_X)<\delta$ and that $d(T^n(\tilde{x}),0_X)>\eps$. Let
	\[
	y_j :=
	\begin{cases}
		x_j+\tilde{x} & \text{ if } d(T^n(x_1),T^n(x_j)) \leq \tfrac{\eps}{2}, \\[5pt]
		x_j & \text{ if } d(T^n(x_1),T^n(x_j)) > \tfrac{\eps}{2},
	\end{cases}
	\quad \text{ for each } 1 \leq j \leq N.
	\]
	For the points $y_1,y_2,...,y_N \in X$, note that $d(x_j,y_j) \leq \max\{ d(x_j,x_j) , d(x_j,x_j+\tilde{x}) \} = d(\tilde{x},0_X) < \delta$ for all $1 \leq j \leq N$. Moreover, for each $1 \leq j \leq N$ we have that
	\[
	d(T^n(x_1),T^n(y_j)) =
	\begin{cases}
		d(T^n(x_1),T^n(x_j+\tilde{x})) & \text{ if } d(T^n(x_1),T^n(x_j)) \leq \tfrac{\eps}{2}, \\[5pt]
		d(T^n(x_1),T^n(x_j)) & \text{ if } d(T^n(x_1),T^n(x_j)) > \tfrac{\eps}{2}.
	\end{cases}
	\]
	Since $d(T^n(x_1),T^n(x_j+\tilde{x})) \geq d(T^n(\tilde{x}),0_X)-d(T^n(x_1),T^n(x_j))$ by translation-invariance and the triangle inequality, we deduce that $d(T^n(x_1),T^n(y_j)) > \tfrac{\eps}{2}$ for all $1 \leq j \leq N$.
\end{proof}

We finally obtain the previously announced generalization of \cite[Theorem~3.2]{BernardesPeRo2017_IEOT_set-valued} and \cite[Theorem~3]{MartinezPeRo2021_MAT_chaos}:

\begin{theorem}\label{The:operators}
	Let $T:X\longrightarrow X$ be an operator acting on a Fr\'echet space $(X,d)$. Hence:
	\begin{enumerate}[{\em(a)}]
		\item If $\lspan(NS(T))$ is dense in $X$, then the following statements are equivalent:
		\begin{enumerate}[{\em(i)}]
			\item the system $(X,T)$ is not equicontinuous;
			
			\item the system $(X,T)$ is U-LYC for the metric $d$;
			
			\item the system $(\Kc(X),\com{T})$ is U-LYC for the metric $d_H$;
			
			\item the system $(\Fc_{\infty}(X),\fuz{T})$ is U-LYC;
			
			\item the system $(\Fc_{0}(X),\fuz{T})$ is U-LYC;
			
			\item the system $(\Fc_{S}(X),\fuz{T})$ is U-LYC;
			
			\item the system $(\Fc_{E}(X),\fuz{T})$ is U-LYC.
		\end{enumerate}
		
		\item If the set $X_0 := \{ x \in X \ ; \ T^{n_k}x \to 0_X \text{ as } k \to \infty \}$ is dense in $X$ for some increasing sequence of positive integers $(n_k)_{k\in\NN} \in \NN^{\NN}$, then the following statements are equivalent:
		\begin{enumerate}[{\em(i)}]
			\item the system $(X,T)$ is not equicontinuous;
			
			\item the system $(X,T)$ is CD-LYC for the metric $d$;
			
			\item the system $(\Kc(X),\com{T})$ is CD-LYC for the metric $d_H$;
			
			\item the system $(\Fc_{\infty}(X),\fuz{T})$ is CD-LYC;
			
			\item the system $(\Fc_{0}(X),\fuz{T})$ is CD-LYC;
			
			\item the system $(\Fc_{S}(X),\fuz{T})$ is U-LYC;
			
			\item the system $(\Fc_{E}(X),\fuz{T})$ is U-LYC.
		\end{enumerate}
	\end{enumerate}
\end{theorem}
\begin{proof}
	(a): Assume (i). Using \cite[Theorem~15]{BernardesBoMuPe2015_ETDS_Li-Yorke} we have that $(X,T)$ is LYC. Thus, by \cite[Theorem~9]{BernardesBoMuPe2015_ETDS_Li-Yorke} we have that $(X,T)$ admits an irregular vector, i.e.\ there exists some $x \in X$ fulfilling that
	\[
	(T^n(x))_{n\in\NN} \text{ is Fr\'echet-unbounded} \quad \text{ while } \quad \liminf_{n\to\infty} d(T^n(x),0_X) = 0.
	\]
	It follows that $S := \{ \lambda x \ ; \ \lambda \in \ ]0,1[ \}$ is an uncountable LY $\eps$-scrambled set for $f$ and $d$, whenever we choose a small enough $\eps>0$ to verify that $d$ is not bounded by $\eps$. We have obtained that (i)~$\Rightarrow$~(ii).~The implications (ii) $\Rightarrow$ (iii), (iv), (v), (vi), (vii) follow from parts (f), (g) and (h) of Theorem~\ref{The:scrambled}. Finally, the implications (iii), (iv), (v), (vi), (vii) $\Rightarrow$ (i) follow from Lemma~\ref{Lem:equicontinuous}. Actually, assume that either the system $(\Kc(X),\com{f})$ or the system $(\Fc(X),\fuz{T})$ admits a LY pair for $d_H$ or for some $\rho \in \{ d_{\infty} , d_{0} , d_{S} , d_{E} \}$, and denote such LY pair by $(K,L) \in \Kc(X)\times\Kc(X)$ or by $(u,v) \in \Fc(X)\times\Fc(X)$ respectively. Thus, if the system $(X,T)$ was equicontinuous, an application of Lemma~\ref{Lem:equicontinuous} would imply the contradiction
	\[
	0 < \limsup_{n\to\infty} d_H(\com{T}^n(K),\com{T}^n(L)) \leq \limsup_{n\to\infty} \left( d_H(\com{T}^n(K),\{0_X\}) + d_H(\com{T}^n(L),\{0_X\}) \right) = 0,
	\]
	or, by statement (a) of Proposition~\ref{Pro:fuzzy.metrics}, the contradiction
	\[
	0 < \limsup_{n\to\infty} \rho(\fuz{T}^n(u),\fuz{T}^n(v)) \leq \limsup_{n\to\infty} \left( d_{\infty}\left(\fuz{T}^n(u),\chi_{\{0_X\}}\right) + d_{\infty}\left(\fuz{T}^n(v),\chi_{\{0_X\}}\right) \right) = 0.
	\]
	
	(b): Since we have that $X_0 \subset NS(T)$, the set $\lspan(NS(T))$ is also dense in this part (b). Thus, the equivalences (i) $\Leftrightarrow$ (vi) $\Leftrightarrow$ (vii) follow from the already proved part (a) of this Theorem~\ref{The:operators}, but also the implications (iii), (iv), (v) $\Rightarrow$ (i) because the notion of CD-LYC is stronger than U-LYC.
	
	In its turn, since it can be easily checked that $X_0^2 \subset \Prox(f,d)$, the system $(X,T)$ has a dense proximal relation and (i)~$\Leftrightarrow$~(ii) holds by Lemma~\ref{Lem:collective} and \cite[Proposition~3.5]{JiangLi2025_JMAA_chaos}. To conclude we will prove the implications (i) $\Rightarrow$ (iii), (iv), (v). Actually, if~(i) holds:
	\begin{enumerate}[(1)]
		\item It can be checked that $X_0^{2N} \subset \Prox(f_{(N)},d_{(N)})$ for all $N \in \NN$, so that $(X^N,f_{(N)})$ has a dense proximal relation for every $N \in \NN$. Using statements (a) and (b) of Theorem~\ref{The:proximality} we deduce that $(\Kc(X),\com{T})$ for $d_H$, and the systems $(\Fc_{\infty}(X),\fuz{T})$ and $(\Fc_{0}(X),\fuz{T})$, have a dense proximal relation.
		
		\item We have that $(X,T)$ is collectively sensitive for $d$ by Lemma~\ref{Lem:collective} so that $(\Kc(X),\com{T})$ for $d_H$, and the systems $(\Fc_{\infty}(X),\fuz{T})$ and $(\Fc_{0}(X),\fuz{T})$, are sensitive by statement (a) of Theorem~\ref{The:sensitivity}.
	\end{enumerate}
	Statements (iii), (iv) and (vi) follow now, because the metric spaces $(\Kc(X),d_H)$, $\Fc_{\infty}(X)$ and $\Fc_{0}(X)$ are completely metrizable (see \cite[Exercise~2.15]{IllanesNad1999_book_hyperspaces}, \cite[Theorem~3.8]{JooKim2000_FSS_the} and \cite[Proposition~4.6]{Huang2022_FSS_some}), by applying \cite[Proposition~3.5]{JiangLi2025_JMAA_chaos} to each of the systems $(\Kc(X),\com{T})$, $(\Fc_{\infty}(X),\fuz{T})$ and $(\Fc_{0}(X),\fuz{T})$.
\end{proof}

From the proof of Theorem~\ref{The:operators} one should note that, in both parts (a) and (b), we could have added the equivalent condition ``{\em the system $(X,T)$ is LYC for the metric $d$}''. Thus, Theorem~\ref{The:operators} is a generalization of \cite[Theorem~3.2]{BernardesPeRo2017_IEOT_set-valued} and \cite[Theorem~3]{MartinezPeRo2021_MAT_chaos}. In particular, Theorem~\ref{The:operators} applies to the examples of weighted shifts considered in \cite[Section~3]{BernardesPeRo2017_IEOT_set-valued}. As a new application, let us recall that an operator $T:X\longrightarrow X$ is called {\em hypercyclic} when it admits a dense $T$-orbit. We have the next:

\begin{corollary}
	Let $T:X\longrightarrow X$ be an operator acting on a Fr\'echet space $(X,d)$. If the operator $T$ is hypercyclic, then the dynamical systems $(\Kc(X),\com{T})$ for $d_H$, $(\Fc_{\infty}(X),\fuz{T})$ and $(\Fc_{0}(X),\fuz{T})$, are CD-LYC, and the dynamical systems $(\Fc_{S}(X),\fuz{T})$ and $(\Fc_{E}(X),\fuz{T})$ are U-LYC. 
\end{corollary}
\begin{proof}
	Let $x_0 \in X$ be a vector with dense $T$-orbit. Let $(n_k)_{k\in\NN}$ such that $T^{n_k}(x_0) \to 0_X$ as $k \to \infty$ and note that, by the continuity of $T$, we have that $X_0 := \{ T^n(x_0) \ ; \ n \in \NN \}$ is a dense set in $X$ fulfilling that $T^{n_k}(x) \to 0_X$ as $k \to \infty$ for every $x \in X_0$. Since hypercyclicity implies that $(X,T)$ is not equicontinuous, this corollary follows automatically from statement (b) of Theorem~\ref{The:operators}.
\end{proof}

Using the theory developed in \cite{JiangLi2025_JMAA_chaos}, similar results may be obtained when $(X,T)$ is only assumed to be an endomorphism of a completely metrizable group. We leave the details to the reader.

\section{Conclusions}\label{Sec_6:conclusions}

In this paper we have carried out a detailed study of the main variants of Li-Yorke chaos for the fuzzy extension $(\Fc(X),\fuz{f})$. Along the paper we have shown that these chaotic-type properties transfer naturally from $(X,f)$ to $(\Kc(X),\com{f})$ and from $(\Kc(X),\com{f})$ to $(\Fc(X),\fuz{f})$, while the converse directions are either false or require particular conditions (such as completeness or linearity). Combining the results obtained here with those in \cite{Lopez2026_IJFS_topological-I,Lopez2026_JIA_topological-II}, we have now achieved a systematic description of the dynamical behaviour of the fuzzy extension $(\Fc(X),\fuz{f})$ as defined in this framework, that is, with $\Fc(X)$ denoting the {\em space of normal fuzzy sets}. These three works together offer a coherent theory unifying most of the dynamical properties that have been considered for these fuzzy systems in the literature.

More precisely, let us recall that the main notions of chaos studied in the literature for the fuzzy dynamical system $(\Fc(X),\fuz{f})$ are \textit{Devaney chaos}, \textit{topological $\Ac$-transitivity}, the \textit{specification property}, the \textit{shadowing property} and the notions of \textit{Li-Yorke chaos}. Until \cite{Lopez2026_IJFS_topological-I,Lopez2026_JIA_topological-II} and this paper, these notions have always been studied for the metrics $d_{\infty}$ and $d_{0}$ (see \cite{BartollMaPeRo2022_AXI_orbit,MartinezPeRo2021_MAT_chaos}). As a summary:
\begin{enumerate}[--]
	\item In \cite{Lopez2026_IJFS_topological-I} we proved that \textit{being Devaney chaotic} is equivalent for all the systems $(\Kc(X),\com{f})$, $(\Fc_{\infty}(X),\fuz{f})$, $(\Fc_{0}(X),\fuz{f})$, $(\Fc_{S}(X),\fuz{f})$, and $(\Fc_{E}(X),\fuz{f})$.
	
	\item In \cite{Lopez2026_IJFS_topological-I} we also proved that \textit{being topologically $\Ac$-transitive} is equivalent for all the systems $(\Kc(X),\com{f})$, $(\Fc_{\infty}(X),\fuz{f})$, $(\Fc_{0}(X),\fuz{f})$, $(\Fc_{S}(X),\fuz{f})$, and $(\Fc_{E}(X),\fuz{f})$;
	
	\item In \cite{Lopez2026_IJFS_topological-I} we further proved that \textit{having the specification property} is equivalent for all the systems $(\Kc(X),\com{f})$, $(\Fc_{\infty}(X),\fuz{f})$, $(\Fc_{0}(X),\fuz{f})$, $(\Fc_{S}(X),\fuz{f})$, and $(\Fc_{E}(X),\fuz{f})$.
	
	\item In \cite{Lopez2026_JIA_topological-II} we proved that \textit{having the finite shadowing property} is equivalent for $(X,d)$, $(\Kc(X),\com{f})$, and $(\Fc_{\infty}(X),\fuz{f})$, \textbf{but not for the systems} $(\Fc_{0}(X),\fuz{f})$, $(\Fc_{S}(X),\fuz{f})$, and $(\Fc_{E}(X),\fuz{f})$.
	
	\item In this paper we proved that Li-Yorke chaos and several variants (such as DC1, \DC{1}, DC2, \DC{2} and DC3) transfer from $(X,f)$ to $(\Kc(X),\com{f})$, and from $(\Kc(X),\com{f})$ to $(\Fc(X),\fuz{f})$ endowed with the metrics $d_{\infty}$, $d_{0}$, $d_{S}$ and $d_{E}$, but that the converse implications are not always true.
\end{enumerate}
Hence, our research has led us to reveal new phenomena such as the existence of strong Li-Yorke-type chaotic behaviour for $(\Fc(X),\fuz{f})$ under very weak assumptions on the original system $(X,f)$ or the hyperextension $(\Kc(X),\com{f})$. It follows that the results obtained in this paper significantly extend those in \cite{MartinezPeRo2021_MAT_chaos}, completing the metric framework initiated in \cite{Lopez2026_IJFS_topological-I,Lopez2026_JIA_topological-II}.

As for possible directions for future research, several natural possibilities arise:
\begin{enumerate}[--]
	\item One could investigate other dynamical properties beyond those considered in \cite{Lopez2026_IJFS_topological-I,Lopez2026_JIA_topological-II} or here (such as the existence of invariant measures or other notions, see for instance \cite{AlvarezLoPe2025_FSS_recurrence,Lopez2024_RIMA_invariant}).
	
	\item One could study alternative metrics on $\Fc(X)$, as suggested in \cite[Section~5]{Lopez2026_IJFS_topological-I} and \cite[Section~6]{Lopez2026_JIA_topological-II}, and analyze whether the results obtained here remain valid or need to be reformulated.
	
	\item One could extend the setting to uniform (and not necessarily metrizable) spaces; see \cite{JardonSanSan2023_IJFS_fuzzy,JardonSanSan2026_TA_transitivity}.
	
	\item One could explore other spaces of fuzzy sets, in which the elements included are not necessarily normal, as discussed in \cite[Section~5]{Lopez2026_IJFS_topological-I}.
\end{enumerate}

In addition, we leave here a list of open problems coming from \cite[Pages 74, 75 and 77]{JardonSan2021_IJFS_sensitivity} and from the results that we have obtained in Section~\ref{Sec_4:sensitivity}:

\begin{problem}\label{Prob:last}
	Let $f:X\longrightarrow X$ be a continuous map acting on a metric space $(X,d)$:
	\begin{enumerate}[(a)]
		\item Suppose that the system $(X,f)$ is strongly sensitive for $d$. Is $(\Fc_{E}(X),\fuz{f})$ strongly sensitive?
		
		\item Suppose that the system $(\Kc(X),\com{f})$ is sensitive for $d_H$. Is $(\Fc_{E}(X),\fuz{f})$ sensitive?
		
		\item Suppose that the system $(\Fc_{E}(X),\fuz{f})$ has a dense proximal relation. Has the system $(\Kc(X),\com{f})$ a dense proximal relation for the metric $d_H$?
	\end{enumerate}
\end{problem}

The relevance of Problem~\ref{Prob:last} lies in the possibility of further clarifying the interplay between sensitivity-type properties and the proximal behaviour across the different systems considered in this work. In particular, a positive answer to parts~(a) and~(b) would allow us to strengthen Theorem~\ref{The:sensitivity} by obtaining a more complete and satisfactory equivalence between the notions of sensitivity for the original system $(X,f)$, the system $(\Kc(X),\com{f})$, and the fuzzy system $(\Fc_{E}(X),\fuz{f})$. On the other hand, a positive solution to part~(c) would make it possible to extend Theorems~\ref{The:proximality}~and~\ref{The:cantor} by replacing its current hypotheses with conditions formulated only on the fuzzy system. Notice that, the extension of Theorem~\ref{The:cantor} would be particularly significant, as it would show that checking dynamical behaviour in the weaker endograph metric $d_{E}$ is sufficient to guarantee strong forms of chaos in the hyperspace system $(\Kc(X),\com{f})$. Therefore, resolving these problems would contribute to a deeper understanding of how dynamical properties are transferred and detected among these associated systems, and would enhance the applicability of the fuzzy framework as a tool to study complex dynamics.

\section*{Funding}

The second author was partially supported by
\[
\text{MCIN/AEI/10.13039/501100011033/FEDER, UE, Project PID2022-139449NB-I00.}
\]

\section*{Acknowledgments}

The authors want to thank Salud Bartoll, Nilson Bernardes Jr., F\'elix Mart\'inez-Jim\'enez, Alfred Peris, and Francisco Rodenas, for numerous readings of the manuscript. The second author is also grateful to \'Angel Calder\'on-Villalobos and Manuel Sanchis for insightful discussions on the topic.

{\footnotesize% We may try "\small" with "\\[-5mm]" ...

}

{\footnotesize
$\ $\\
	
\textsc{Illych Alvarez}: Escuela Superior Polit\'ecnica del Litoral, Facultad de Ciencias Naturales y Matem\'aticas, Km.\ 30.5 V\'ia Perimetral, Guayaquil, Ecuador. e-mail: ialvarez@espol.edu.ec
	
\textsc{Antoni L\'opez-Mart\'inez}: Universitat Polit\`ecnica de Val\`encia, Institut Universitari de Matem\`atica Pura i Aplicada, Edifici 8E, 4a planta, 46022 Val\`encia, Spain. e-mail: alopezmartinez@mat.upv.es
}

\end{document}